\documentclass[12pt]{article}

\usepackage[latin1]{inputenc}
\usepackage{t1enc}
\usepackage{amsmath, amssymb, amsfonts, amsthm, amsopn,epsfig}
\usepackage[none]{hyphenat}
\usepackage{tikz}
\usepackage{float}
\usepackage{graphicx}
\usepackage{caption}
\usepackage[margin=1in]{geometry}
\usepackage[toc,page]{appendix}
\usepackage{epstopdf}
\usepackage{etoolbox}
\usepackage[hidelinks]{hyperref}

\theoremstyle{plain}
\newtheorem{theorem}{Theorem}

\newtheorem{claim}[theorem]{Claim}
\newtheorem{lemma}[theorem]{Lemma}
\newtheorem{conjecture}[theorem]{Conjecture}

\newtheorem*{plemma}{Partition Lemma}

\newcommand{\Mod}[1]{\ (\mathrm{mod}\ #1)}

\theoremstyle{definition}

\DeclareMathOperator*{\PV}{\mathbb{P}}
\DeclareMathOperator*{\PP}{\mathcal{P}}

\author{Istv\'{a}n Tomon\thanks{\'{E}cole Polytechnique F\'{e}d\'{e}rale de Lausanne, Research partially supported by Swiss National Science Foundation grants no. 200020-162884 and 200021-175977.			
		\emph{e-mail}: \textbf{istvan.tomon@epfl.ch}}
}
\title{Packing the Boolean lattice with copies of a poset}
\date{}

\begin{document}
\sloppy
\maketitle

\begin{abstract}
	Let $P$ be a partially ordered set. We prove that if $n$ is sufficiently large, then there exists a packing $\mathcal{P}$ of copies of $P$ in the Boolean lattice $(2^{[n]},\subset)$ that covers almost every element of $2^{[n]}$: $\mathcal{P}$ might not cover the minimum and maximum of $2^{[n]}$, and at most $|P|-1$ additional points due to divisibility. In particular, if $|P|$ divides $2^{n}-2$, then the truncated Boolean lattice $2^{[n]}-\{\emptyset,[n]\}$ can be partitioned into copies of $P$. This confirms a conjecture of Lonc from 1991.
%	 Also, we present a short proof of the following theorem. If $P$ has a unique minimum and maximum, $|P|$ is a power of $2$, and $n=\Omega(|P|^{8})$, then the Boolean lattice $2^{[n]}$ can be partitioned into copies of $P$. This improves a recent result of Gruslys, Leader and Tomon, in which the same statement is proved for $n= 2^{|P|^{\Omega(1)}}$.
\end{abstract}

\section{Introduction}

Let $P$ and $Q$ be posets (partially ordered sets). A subset $Q'$ of $Q$ is a \emph{copy} of $P$ if the subposet of $Q$ induced on $Q'$ is isomorphic to $P$. A \emph{$P$-packing} of $Q$ is a family of disjoint copies of $P$ in $Q$, and a \emph{$P$-partition} is a $P$-packing that covers every element of $Q$.

The Boolean lattice $2^{[n]}$ is the partially ordered set on the power set of $[n]=\{1,...,n\}$, in which the ordering is given by the inclusion relation. Lonc proved \cite{L}, settling a conjecture of Sands \cite{S} and Griggs \cite{Gri}, that if $P$ is a chain of size $h$ and $n\geq 2^{2^{36h^{2}}}$, then $2^{[n]}$ has a $P$-packing that covers all but at most $h-1$ elements of $2^{[n]}$. In particular, if $h$ is a power of $2$, then $2^{[n]}$ has a $P$-partition. The bound on $n$ was improved by the author of this paper \cite{T}: the assumption $n>500h^{2}$ is enough, and this bound is the best possible up to the constant factor.

 Lonc \cite{L} conjectured two natural extensions of his result, where chain is replaced with an arbitrary poset $P$. Clearly, if $2^{[n]}$ has a $P$-partition for some positive integer $n$, then the size of $P$ must be a power of $2$, and $P$ must have a unique minimum and maximum\footnote{otherwise, there is no copy of $P$ covering $[n]$ or $\emptyset$ in $2^{[n]}$}. The first conjecture states that these conditions are also sufficient to guarantee the existence of a $P$-partition in $2^{[n]}$ for $n$ sufficiently large. This conjecture was recently verified by Gruslys, Leader and Tomon \cite{GLT2} for $n=2^{|P|^{\Omega(1)}}$. We give a new, simpler proof of this result which also improves the bound on $n$.
 
 \begin{theorem}\label{mainthm0}
 	Let $P$ be a poset with a unique minimum and maximum, and size $2^{k}$. If $n\geq c|P|^{8}$, where $c$ is some absolute constant, then $2^{[n]}$ can be partitioned into copies of $P$.
 \end{theorem}

The second conjecture of Lonc is concerned with posets $P$ which do not necessarily satisfy that the size of $P$ is a power of $2$, or have a unique minimum and maximum. In this case,  it is still reasonable to believe that there exists a $P$-packing in $2^{[n]}$ that covers almost everything. More precisely, the conjecture states that if $n$ is sufficiently large and $2^{n}-2$ is divisible by $|P|$, then the truncated Boolean lattice $2^{[n]}-\{\emptyset,[n]\}$ has a $P$-partition. This conjecture was verified by Lonc \cite{L2} in the case $P$ is an antichain, or $|P|\leq 4$. Also, Gruslys, Leader and Tomon \cite{GLT2} proposed a relaxation of this conjecture. That is, there exists a constant $c(P)$ such that $2^{[n]}$ has a $P$-packing that covers all but at most $c(P)$ elements of $2^{[n]}$ for every $n$. This conjecture was verified by the author of this paper \cite{T2} in case $P$ has a unique minimum and maximum (but size not necessarily a power of $2$). 

We settle both of the aforementioned conjectures in the following theorem.

\begin{theorem}\label{mainthm2}
	Let $P$ be a poset. There exists $n_{0}=n_{0}(P)$ such that if $n\geq n_{0}$, then there exists a $P$-packing $\mathcal{P}$ of $2^{[n]}-\{\emptyset,[n]\}$  such that the number of elements not covered by $\mathcal{P}$ is  at most $|P|-1$.
\end{theorem}

Our paper is organized as follows. In the next subsections, we discuss some related partitioning results and we introduce our notation. In Section \ref{sect:thm0}, we prove Theorem \ref{mainthm0}. In Section \ref{sect:thm2}, we prove Theorem \ref{mainthm2}. We finish our paper with some remarks and open problems in Section \ref{sect:remarks}.

\subsection{Related work}

Recently, problems in which the goal is to partition certain product structures, such as $\mathbb{Z}^{n}$ \cite{GLT1}, the graph of the hypercube \cite{Gru,GL}, the Boolean lattice \cite{GLT2,T2}, graph powers of cycles \cite{BMS}, into copies\footnote{the definition of copy varies according to the structure we are interested in} of some set gained increased interest due to a remarkable result of Gruslys, Leader and Tan \cite{GLT1}. Let us briefly outline their main lemma.

\begin{plemma}Let $X$ be a finite set and let $\mathcal{F}$ be a family of subsets of $X$. Suppose that there exist a positive integer $t$, and two multisets $\mathcal{F}_{1},\mathcal{F}_{2}$ of elements of $\mathcal{F}$ such that every $x\in X$ is covered by exactly $t$ elements of $\mathcal{F}_{1}$, and $1\Mod t$ elements of $\mathcal{F}_{2}$. Then for sufficiently large $n$, the cartesian power $X^{n}$ can be partitioned into copies of elements of $\mathcal{F}$. Here, $F'\subset X^{n}$ is a copy of  $F\in \mathcal{F}$ if $F'$ can be written as $$\{x_{1}\}\times...\times \{x_{i-1}\}\times F\times\{x_{i+1}\}\times\{x_{n} \},$$
for some $i\in [n]$ and $x_{1},\dots,x_{i-1},x_{i+1},\dots,x_{n}\in X$.
\end{plemma}

 While this lemma proved to be powerful in partitioning problems, we were unable to adapt it for "almost partitioning" problems. That is, for problems in which there exists no desired partition for some obvious reason (for example divisibility is not satisfied), but we are still hoping to find a packing that covers almost every element. See \cite{GL,T2} for such results. Therefore, in this paper we develop our own packing method to deal with Theorem \ref{mainthm0} and Theorem \ref{mainthm2}, which exploits the flexibility of copies of posets, and we do not utilize the previously described lemma in any way.

\subsection{Preliminaries and notation}\label{sect:prelim}
If $m\leq n$ are integers, let $[m,n]=\{m,m+1,...,n\}$.

If $(P,\leq_{P})$ is a partially ordered set, we may write simply $P$ when referring to $(P,\leq_{P})$. In this paper, every set is endowed with at most one partial order, so this should not lead to any confusion. Also, we use $\leq$ instead of $\leq_{P}$ if it is clear from the context which poset is under consideration.

 If $(P_{1},\leq_{1}),\dots,(P_{k},\leq_{k})$ are partially ordered sets, then the cartesian product $P_{1}\times\dots\times P_{k}$ is endowed with the pointwise ordering $\leq$:  $(x_{1},\dots,x_{k})\leq(y_{1},\dots,y_{k})$ if $x_{i}\leq_{i} y_{i}$ for $i\in [k]$. A \emph{grid} is a poset isomorphic to the cartesian product $[a_{1}]\times\dots\times [a_{k}]$, where $[a_{i}]$ is endowed with the natural total order. We use the signs $\prec$ or $\preceq$ to denote comparabilities between elements of a grid. The \emph{size} of the grid $[a_{1}]\times...\times [a_{k}]$ is the formal expression $a_{1}\times...\times a_{k}$. Note that $2^{[n]}\cong[2]^{n}$.

The \emph{dimension} (Dushnik-Miller dimension) of a poset $P$ is the smallest positive integer $d$ such that the grid $[k]^{d}$ contains a copy of $P$ for some $k$. Also, this is equal to the smallest number $d$ for which there exist $d$ bijections $\pi_{1},\dots,\pi_{d}:P\rightarrow [|P|]$ such that for every $p,q\in P$,  $p\leq_{P} q$ if and only if $\pi_{i}(p)\leq \pi_{i}(q)$ for $i\in [d]$.

If $A$ is a subset of $B$, we write $B-A$ instead of $B\setminus A$, and if $A$ is a one element set $\{a\}$, then we write $A-a$ instead of $A-\{a\}$.

Finally, we shall work with multiple levels of containment, especially in Section \ref{sect:thm2}. To avoid confusion, we refer to elements of $[n]$ as \emph{base elements}, subsets of $[n]$ as just \emph{elements} or \emph{sets}, subsets of $2^{[n]}$ as \emph{families}, and subsets of $2^{2^{[n]}}$ as \emph{collections}. For example, a copy of $P$ in $2^{[n]}$ is a family, while a $P$-packing is a collection.

\section{Posets with unique minimum and maximum}\label{sect:thm0}

Instead of Theorem \ref{mainthm0}, we shall prove the following slightly stronger result.

\begin{theorem}\label{mainthm1}
	Let $P$ be a poset with a unique minimum and maximum, size $2^{k}$ and dimension $d$. If $n=\Omega(d^{4}|P|^{4})$, then $2^{[n]}$ can be partitioned into copies of $P$.
\end{theorem}

 By a result of Hiraguchi \cite{Hi}, we have $d\leq |P|/2$ for every poset $P$, so Theorem \ref{mainthm1} truly implies Theorem \ref{mainthm0}. The proof of Theorem \ref{mainthm1} can be outlined in $3$ simple steps. Let $\frac{[h]^{d}}{[h]^{d}}$ denote the poset which is the union of two disjoint copies of $[h]^{d}$, $H_{0}$ and $H_{1}$, and every element of $H_{1}$ is larger than every element of $H_{0}$. 

\noindent
Step 1: We prove that if $h=\Omega(d|P|^{2})$ and $|P|$ divides $h$, then $\frac{[h]^{d}}{[h]^{d}}$ has a $P$-partition. 

\noindent
Step 2: We show that if $m=\Omega(d^{2})$, then $[2h]^{m}$ has an $\frac{[h]^{d}}{[h]^{d}}$-partition.

\noindent
Step 3: We conclude by showing that if $n=\Omega(h^{2}m)$ and $h$ is a power of $2$, then $2^{[n]}$ has a $[2h]^{m}$-partition.

\bigskip

Now let us start with Step 1. Let $H_{0},H_{1}$ be the two copies of $[h]^{d}$ forming $\frac{[h]^{d}}{[h]^{d}}$, where every element of $H_{1}$ is larger then every element of $H_{0}$. Roughly saying, first, we find a dense $P$-packing $\mathcal{P}_{i}$ in $H_{i}$ for $i=0,1$, this is done in Claim \ref{densepacking}. Then, we construct a dense $P$-packing  $\mathcal{M}_{1}$ in the set of minimums of $\mathcal{P}_{1}$, and a dense $P$-packing $\mathcal{M}_{0}$ in the maximums of $\mathcal{P}_{0}$. Then, we move the maximums of the members of  $\mathcal{P}_{0}$ covered by $\mathcal{M}_{0}$ to fill up uncovered elements of $H_{1}$, and move minimums of members of $\mathcal{P}_{1}$ covered by $\mathcal{M}_{1}$ to fill uncovered elements of $H_{1}$. This is done in Claim \ref{pair}.

\begin{figure}
	\begin{center}
		\includegraphics[scale=0.5]{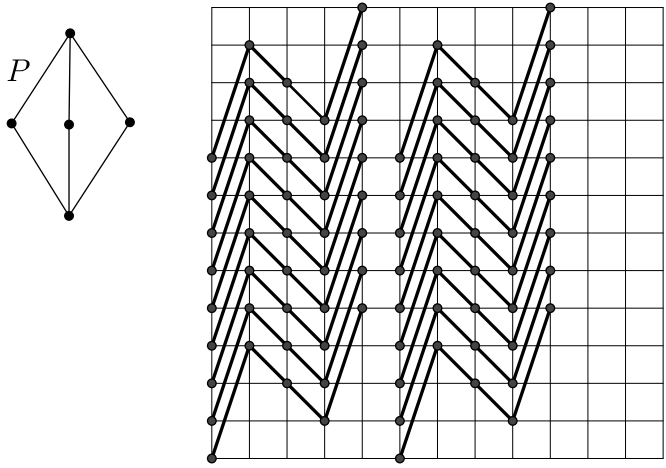}
	\end{center}
	\caption{A dense packing of $[13]\times [13]$ with copies of the $5$ element poset $P$.}
	\label{image}
\end{figure}

\begin{claim}\label{densepacking}
Let $P$ be a poset with a unique minimum and maximum, and let $d$ be the dimension of $P$. Also, let $h_{1},\dots, h_{d}$ be positive integers and let $H=[h_{1}]\times\dots\times [h_{d}]$. Then $H$ contains a $P$-packing $\mathcal{P}$ with the following properties:

(1) the members of $\mathcal{P}$ cover at least $(h_{1}-|P|)\dots(h_{d}-|P|)$ elements,

(2)	the minimum and maximum elements of the members of $\mathcal{P}$ form two grids, each of size $$\lfloor h_{1}/|P|\rfloor\times(h_{2}-|P|)\times\dots\times (h_{d}-|P|).$$
\end{claim}	

\begin{proof}
	By the defnition of dimension, there exist $d$ bijections $\pi_{1},\dots,\pi_{d}:P\rightarrow [|P|]$ such that for any $p,q\in P$, we have $p\leq q$ iff $(\pi_{1}(p),\dots,\pi_{d}(p))\preceq (\pi_{1}(q),\dots,\pi_{d}(q))$. If $m_{0}$ and $m_{1}$ are the unique minimum and maximum of $P$, respectively, then ${\pi_{1}(m_{0})=\dots=\pi_{d}(m_{0})=1}$ and ${\pi_{1}(m_{1})=\dots=\pi_{d}(m_{1})=|P|}$.
	
	For $(a_{1},\dots,a_{d})\in [\lfloor h_{1}/|P|\rfloor]\times[h_{2}-|P|]\times\dots\times [h_{d}-|P|]$, let $$P_{a_{1},\dots,a_{d}}=\{(\pi_{1}(p)+(a_{1}-1)|P|,\pi_{2}(p)+a_{2},\dots,\pi_{d}(p)+a_{d}):p\in P\}\subset H.$$ 
	We show that $\mathcal{P}=\{P_{a_{1},\dots,a_{d}}:(a_{1},\dots,a_{d})\in [\lfloor h_{1}/|P|\rfloor]\times[h_{2}-|P|]\times\dots\times [h_{d}-|P|]\}$ is a $P$-packing satisfying (1) and (2). See Figure \ref{image} for an illustration of this packing.
	
	First of all, $P_{a_{1},\dots,a_{d}}$ is truly a copy of $P$. Indeed, it is a translate of the set $\{(\pi_{1}(p),\dots,\pi_{d}(p)):p\in P\}$, which is copy of $P$ by definition. Now we show that $\mathcal{P}$ is a packing. Suppose that $P_{a_{1},\dots,a_{d}}$ and $P_{b_{1},\dots,b_{d}}$ intersect, then $$(\pi_{1}(p)+(a_{1}-1)|P|,\pi_{2}(p)+a_{2},\dots,\pi_{d}(p)+a_{d})=(\pi_{1}(q)+(b_{1}-1)|P|,\pi_{2}(q)+b_{2},\dots,\pi_{d}(q)+b_{d}),$$
	for some $p,q\in P$. In particular, $\pi_{1}(p)+(a_{1}-1)|P|=\pi_{1}(q)+(b_{1}-1)|P|$. But $\pi_{1}(p),\pi_{1}(q)\in[|P|]$, so this is only possible if $a_{1}=b_{1}$ and $p=q$. Now if $\pi_{i}(p)+a_{i}=\pi_{i}(q)+b_{1}$, we have $a_{i}=b_{i}$ as well. Thus, $(a_{1},\dots,a_{d})=(b_{1},\dots,b_{d})$.
	
	Now (1) clearly holds, as $\mathcal{P}$ has $\lfloor h_{1}/|P|\rfloor(h_{2}-|P|)\dots(h_{d}-|P|)$ members, each of size $|P|$. Finally, (2) holds as the set of minimums of the members of $\mathcal{P}$ is $$\{1,|P|+1,\dots,|P|\lfloor h_{1}/|P|-1\rfloor+1\}\times \{2,\dots,h_{2}-|P|+1\}\times\dots\times \{2,\dots,h_{d}-|P|+1\},$$
	and the set of maximums is
	$$\{|P|,2|P|,\dots,|P|\lfloor h_{1}/|P|\rfloor\}\times \{|P|+1,\dots,h_{2}\}\times\dots\times \{|P|+1,\dots,h_{d}\}.$$
\end{proof}

\begin{claim}\label{pair}
	Let $P$ be a poset with a unique minimum and maximum, and dimension $d$. Let $h,k, n$ be positive integers such that $h\geq 2d|P|^{2}$ and $h$ is divisible by $|P|$.   Then $\frac{[h]^{d}}{[h]^{d}}$ has a $P$-partition.
\end{claim}

\begin{proof}
	Let $H_{0},H_{1}$ be the copies of $[h]^{d}$ forming $\frac{[h]^{d}}{[h]^{d}}$ such that every element of $H_{1}$ is larger then every element of $H_{0}$. By Lemma \ref{densepacking}, for $i=0,1$, $H_{i}$ has a packing $\mathcal{P}_{i}$ such that $\mathcal{P}_{i}$ covers at least $(h-|P|)^{d}$ elements of $H_{i}$, and the maximums of the members of $\mathcal{P}_{0}$ form a $(h/|P|)\times (h-|P|)\times\dots\times(h-|P|)$ sized grid $M_{0}$, while the minimums of the members of $\mathcal{P}_{1}$ form a $(h/|P|)\times (h-|P|)\times\dots\times(h-|P|)$ sized grid $M_{1}$. Applying Lemma \ref{densepacking} again, $M_{i}$ has a packing $\mathcal{M}_{i}$ covering at least $(h/|P|-|P|)(h-2|P|)^{d-1}$ elements. 
	
	Let $A_{i}$ be the set of elements covered by the members of $\mathcal{P}_{i}$ and let $B_{i}=H_{i}\setminus A_{i}$. We have 
	$|A_{i}|\geq (h-|P|)^{d}\geq h^{d}(1-d|P|/h)$, so $|B_{i}|\leq dh^{d-1}|P|$.
	 Also, let $C_{i}$ be the set of elements covered by the members of $\mathcal{M}_{i}$, then $$|C_{i}|\geq \left(\frac{h}{|P|}-|P|\right)(h-2|P|)^{d-1}>h^{d}\frac{1-2d|P|/h}{|P|}\geq |B_{1-i}|.$$
	 
	 We will modify the members of $\mathcal{P}_{i}$ by moving their minimum or maximum elements to $H_{1-i}$, filling the set $B_{1-i}$. As $|P|$ divides $h$ and $|A_{i}|$, we have that $|P|$ divides $|B_{i}|$ as well. Let $\mathcal{M}'_{i}$ be a subfamily of $\mathcal{M}_{i}$ with exactly $|B_{1-i}|/|P|$ elements, let $C_{i}'\subset H_{i}$ be the set of elements covered by the members of  $\mathcal{M}_{i}'$ and let $\phi_{i}:C_{i}'\rightarrow B_{1-i}$ be an arbitrary bijection. 
	 
	 Let $S\in \mathcal{P}_{0}$. If the maximum $x$ of $S$ is covered by $\mathcal{M}'_{0}$, then let $S'=(S\setminus \{x\})\cup\{\phi_{0}(x)\}$. Otherwise, let $S'=S$.  Then $S'$ is also a copy of $P$ as $x\prec \phi_{0}(x)$. Set $\mathcal{P}_{0}'=\{S':S\in\mathcal{P}_{0}\}$. Then $\mathcal{P}_{0}'$ is a $P$-packing, and the set of elements covered by the members of $\mathcal{P}_{0}'$ is $(A_{0}\setminus C_{0}')\cup B_{1}$
	 
	 We define $\mathcal{P}'_{1}$ similarly, replacing maximum with minimum, which results in a $P$-packing covering ${(A_{1}\setminus C_{1}')\cup B_{0}}$. But then $$\mathcal{P}=\mathcal{P}'_{1}\cup\mathcal{P}'_{2}\cup\mathcal{M}'_{1}\cup\mathcal{M}'_{2}$$
	 is a $P$-partition of $H_{0}\cup H_{1}$.
\end{proof}

Now we can move to Step 2. Roughly, our idea to partition $[2h]^{m}$ into copies of $\frac{[h]^{d}}{[h]^{d}}$ is as follows. We partition $[2h]^{m}$ into $2^{m}$ copies of $[h]^{m}$ in the obvious way. For each copy of $[h]^{m}$, we pick $d$ coordinates and partition $[h]^{m}$ into copies of $[h]^{d}$ along these $d$ coordinates. We pick our $d$ coordinates with the help of Claim \ref{matching}, so that we can match the copies of $[h]^{d}$ to form a copy of $\frac{[h]^{d}}{[h]^{d}}$. This is done in Claim \ref{grid}. 

In the proofs, we need the following result of the author of this paper \cite{T} mentioned in the Introduction.

\begin{lemma}[\cite{T}]\label{longchain}
	Let $h$ be a power of $2$. If $n=\Omega(h^{2})$, then $2^{[n]}$ can be partitioned into chains of size $h$.
\end{lemma}

\begin{claim}\label{matching}
	Let $d$ be a positive integer. If $m\geq \Omega(d^{2})$, then there is a matching $M$ in $2^{[m]}$ such that if $\{x,y\}\in M$, then $x$ and $y$ are comparable and $|x\triangle y|\geq d$.
\end{claim}

\begin{proof}
   Let $k$ be a positive integer such that $d\leq 2^{k}< 2d$. By Lemma \ref{longchain}, if $m= \Omega(d^{2})$, then $2^{[m]}$ has a partition $\mathcal{C}$ into chains of size $2^{k+1}$. For each chain $C\in \mathcal{C}$, if $c_{1}\subset \dots \subset c_{2^{k+1}}$ are the elements of $C$, then match $c_{i}$ with $c_{i+2^{k}}$ for $i\in [2^{k}]$. Clearly, $c_{i}\subset c_{i+2^{k}}$ and $|c_{i}\triangle c_{i+2^{k}}|\geq 2^{k}\geq d$.
\end{proof}

Let us remark that the bound in Claim \ref{matching} is the best possible up to the constant factor. Indeed, every element of $2^{[m]}$ lying in the $d$ middle levels is matched to the rest of the elements. Therefore, we must have that the number of elements of $2^{[m]}$ in the $d$ middle levels is at most $2^{m}/2$. But then by well known  concentration inequalities, $m=\Omega(d^{2})$.

\begin{claim}\label{grid}
	 Let $h,d,m$ be positive integers such that $m=\Omega(d^{2})$. Then $[2h]^{n}$ can be partitioned into copies of $\frac{[h]^{d}}{[h]^{d}}$.
\end{claim}

\begin{proof}
	Let $M$ be a matching of $[2]^{m}$ such that if $\{x,y\}\in M$, then $x$ and $y$ are comparable and $|x-y|\geq d$. (Such a matching exists by Claim \ref{matching}.) For each $x=(x_{1},\dots,x_{n})\in [2]^{n}$, let $$H_{x}=\{(a_{1},\dots,a_{n}):\forall i\in [n],(x_{i}-1)h<a_{i}\leq x_{i} h\}.$$
	Then $\{H_{x}\}_{x\in [2]^{n}}$ is partition of $[2h]^{n}$ into grids isomorphic to $[h]^{n}$.
	Suppose that $x=(x_{1},\dots,y_{n})\in [2]^{n}$ is matched with $y=(y_{1},\dots,y_{n})$ in $M$. We show that $H_{x}\cup H_{y}$ can be partitioned into copies of  $\frac{[h]^{d}}{[h]^{d}}$. Clearly, if this is true for every pair in $M$, we are done.
	
	 Without loss of generality, suppose that $x\prec y$ and that $x$ and $y$ differ in the last $d$ coordinates. For $z=(z_{1},\dots,z_{n-d})\in [h]^{n-d}$, let 
	 $$H_{x}(z)=\{(z_{1}+(x_{1}-1)h,\dots,z_{n-d}+(x_{n-d}-1)h,a_{n-d+1},\dots,a_{n}):a_{n-d+1},\dots,a_{n}\in [h]^{d}\},$$
	 and 
	 $$H_{y}(z)=\{(z_{1}+(y_{1}-1)h,\dots,z_{n-d}+(y_{n-d}-1)h,a_{n-d+1},\dots,a_{n}):a_{n-d+1},\dots,a_{n}\in [h+1,2h]^{d}\}.$$
	 Then the sets $H_{x}(z)$ and $H_{y}(z)$ are isomorphic to the grid $[h]^{d}$ and form a partition of $H_{x}$ and $H_{y}$, respectively. Moreover, every element of $H_{x}(z)$ is $\prec$-less than every element of $H_{y}(z)$. But then $H_{x}(z)\cup H_{y}(z)$ is a copy of  $\frac{[h]^{d}}{[h]^{d}}$.
\end{proof}

We conclude this section with Step 3, and the proof of Theorem \ref{mainthm1}.

\begin{proof}[Proof of Theorem \ref{mainthm1}]
Let $h$ be a power of $2$ such that  $2d|P|^{2}\leq h<4d|P|^{2}$. Let $s$ be the smallest number for which $2^{[s]}$ can be partitioned into chains of size $2h$, and let $\mathcal{C}$ be such a chain partition. By Lemma \ref{longchain},  $s=\Omega(h^{2})$. Let $m$ be the smallest positive integer for which $[2h]^{m}$ has a  $\frac{[h]^{d}}{[h]^{d}}$-partition, and therefore a $P$-partition by Claim \ref{pair}. Then $m=O(d^{2})$ by Claim \ref{grid}.

Set $n=ms$. Then $2^{[n]}$ is isomorphic to the cartesian power $(2^{[s]})^m$. But $\mathcal{C}$ generates a partition of $(2^{[s]})^m$ into the cartesian products $C_{1}\times...\times C_{m}$, where $C_{1},\dots,C_{m}\in\mathcal{C}$. Moreover, $C_{1}\times...\times C_{m}$ is isomorphic to $[2h]^{l}$, which means that it has a $P$-partition. The union of these $P$-partitions is a $P$-partition of $2^{[n]}$.

 Also, we have $n=sm=O(h^{2}d^{2})=O(|P|^{4}d^{4})$. We finish the proof by noting that if $2^{[n]}$ has a $P$-partition, then so does $2^{[n']}$ for $n'>n$.
\end{proof}

\section{General posets}\label{sect:thm2}

In this section, we prove Theorem \ref{mainthm2}. Denote the truncated Boolean lattice $2^{[n]}-\{\emptyset,[n]\}$ by $T(n)$.

\subsection{Overview of the proof}

Let us take a quick look back at the proof of Theorem \ref{mainthm0}. If $P$ has a unique minimum and maximum and size $2^{k}$, then we were able to find a nice poset $Q$, namely $\frac{[h]^{m}}{[h]^{m}}$, which has the property (if the parameters are set correctly) that $Q$ can be "easily" partitioned into copies of $P$, and $2^{[n]}$ can be "easily" partitioned into copies of $Q$.

Now let us suppose that $P$ is an arbitrary poset. We would like to follow a similar train of thoughts as above. Unfortunately, $\frac{[h]^{m}}{[h]^{m}}$ does not have advantageous properties regarding $P$-partitions anymore. Instead, we define a slightly more complicated poset $A$, which we call \emph{absorber}, with the property that for any $R\subset A$, where $|R|$ is not too large, $A-R$ has a $P$-partition provided $|P|$ divides $|A-R|$. This can be found in Section \ref{sect:absorber}. We use the name absorber, as these families have similar properties as the graph absorbers used in graph partitioning problems. See the seminal paper of Erd\H{o}s, Gy\'{a}rf\'{a}s and Pyber \cite{EGP} for one of the first applications of the so called absorption method (note that, however, the terms absorber/absorption were coined later).

The next natural idea would be to construct an $A$-packing of $2^{[n]}$ that covers all but a small number of elements. However, we are unable to do this. Instead, we construct a dense $A$-packing of $2^{[n]}$ such that every uncovered element $x$, which is not too close to $\emptyset$ or $[n]$ in some sense, can be matched to one of the absorbers $A$ so that $A\cup\{x\}$ contains a copy of $P$ covering $x$. Also, we deal with the elements that are too close to $\emptyset$ or $[n]$ separately. This can be found in Section \ref{sect:denseabsorber}.

At this point, we managed to cover every element of $T(n)$ not contained in our collection of absorbers with copies of $P$. One would like to argue that then we are done as we can find a $P$-packing in each of the absorbers that covers the previously not covered elements. However, this is only true if the number of not covered elements in each of the absorbers is divisible by $|P|$. To overcome this problem, we add some extra dimensions and consider $2^{[n'+n]}\cong 2^{[n']}\times 2^{[n]}$, where we drop a few copies of $P$ that correct the divisibility issues. This is prepared in Section \ref{sect:modp}.

We are almost done, the only problem is when we added those extra dimensions, we did not cover the elements $T(n')\times \{\emptyset,[n]\}$. We construct a $P$-packing which deals with this problem in Section \ref{sect:correction}. Finally, in Section \ref{sect:final}, we put all of our $P$-packings together to form the desired $P$-packing of $T([n'+n])$.

\bigskip

The main ideas of the proof are contained in Claim \ref{absorber}, Claim \ref{absorber_packing} and Claim \ref{absorber_matching}, while the other parts of the proof are tying up loose ends (of which there happens to be a lot of), and might be quite technical in nature. Therefore, we advise the interested reader to put more emphasis on understanding the aforementioned claims, and skip the other parts of the proof at first reading.

\subsection{Special elements}

  Let $Q$ be a copy of $P$ in $2^{[n]}$. We say that $f\in [n]$ is \emph{special for} $x\in Q$ if either 

(1) $x$ is minimal in $Q$, $f\in x$, and every $y\in Q$ containing $i$ satisfies $x\subset y$, or

(2) $x$ is maximal in $Q$, $f\not\in x$, and every $y\in Q$ not containing $i$ satisfies $y\subset x$.\\
 Special elements are going to be used to modify certain copies of $P$ by moving one of their minimal or maximal elements. Therefore, minimal (or maximal) elements of copies of $P$ with some special base element will play a similar role as the minimal (or maximal) elements in the proof of Theorem \ref{mainthm0}. To this purpose, we shall exploit the following properties of special elements.

\begin{claim}\label{special}
	 Let $Q\subset 2^{[n]}$ be a copy of $P$.
	
	(1) If $x\in Q$ is minimal and $f\in [n]$ is special for $x$, then for every $x'\in 2^{[n]}$ satisfying $f\in x'\subset x$, the family $(Q-x)\cup\{x'\}$ is also a copy of $P$.
	
	(2)  If $y\in Q$ is maximal and $f\in [n]$ is special for $y$, then for every $y'\in 2^{[n]}$ satisfying $y\subset y'$ and $f\not\in y'$, the family $(Q-y)\cup\{y'\}$ is a copy of $P$.
\end{claim}

\begin{proof}
	(1) Let $y\in Q-x$. If $x\subset y$, then $x'\subset y$ as $x'\subset x$. Also, if $x\not\subset y$, then $f\not\in y$ as $i$ is special for $x$. But $f\in x'$, so $x'\not\subset y$ as well. This means that $x$ and $x'$ are comparable to the same set of elements in $Q$, so the posets $Q$ and $(Q-x)\cup \{x'\}$ are isomorphic.
	
	(2) The proof is similar to the proof of (1).
\end{proof}

\begin{claim}\label{embedding}
	Let $n\geq |P|$ and $i\in [n]$. There exists a copy $Q$ of $P$ in $2^{[n]}$ such that $\{f\}\in Q$ is minimal in $Q$. Also, there exists a copy $Q'$ of $P$ in $2^{[n]}$ such that $[n]-f\in Q'$ is maximal. In particular, $f$ is special for $\{f\}$ in $Q$, and $f$ is special for $[n]-f$ in $Q'$.
\end{claim}

\begin{proof}
	We prove the existence of such $Q$, the existence of $Q'$ can be proved similarly. Let $\pi:P\rightarrow [n]$ be an injection which satisfies that $\pi(p)=f$ for some minimal element $p\in P$. Also, define $\phi:P\rightarrow 2^{[n]}$ such that $\phi(p)=\{\pi(q):q\leq_{P}p\}$. Then $Q=\phi(P)$ is a copy of $P$ in $2^{[n]}$ in which $\phi(p)=\{f\}$ is minimal.
\end{proof}

We shall also use the following immediate corollary of Claim \ref{embedding}.

\begin{claim}\label{manyembedding}
 Let $n\geq |P|$ and $f\in [n]$. There exist at least $2^{n-|P|}$ disjoint copies $Q$ of $P$ in $2^{[n]}$ for which $f$ is special for some minimal element of $Q$. Also, there exist at least $2^{n-|P|}$ disjoint copies $Q'$ of $P$ in $2^{[n]}$ for which $f$ is special for some maximal element of $Q'$.	
\end{claim}

\begin{proof}
	Again, we prove only the first claim, the second claim can be proved in a similar manner. Without loss of generality, suppose that $f=1$. By Claim \ref{embedding}, there exists a copy $Q$ of $P$ in $2^{[|P|]}$ in which $\{1\}$ is a minimal element. For each $z\subset [|P|+1,n]$, let $Q_{z}=\{p\cup z:p\in Q\}$. Then $\{Q_{z}:z\subset [|P|+1,n]\}$ is a collection of $2^{n-|P|}$ disjoint copies of $P$. Also, $\{1\}\cup z$ is a minimal element of $Q_{z}$ for which $1$ is special. 
\end{proof}

\subsection{Absorbers and their properties}\label{sect:absorber}

%A subset $A\subset 2^{[n]}$ is a \emph{$d$-dimensional absorber}, (or just simply an \emph{absorber}, if $d$ is clear from the context,) if it has the following form.

 A \emph{$d$-dimensional absorber} in $2^{[n]}$ (or just simply an \emph{absorber}, if $d$ is clear from the context) is the union of four copies of $2^{[d]}$ positioned in a particular way. To this end, we need to introduce a couple of parameters. Let $\alpha_{1},\alpha_{2},\alpha_{3},\alpha_{4}$ be four disjoint $d$-element subsets of $[n]$ and let ${\beta=[n]-\bigcup_{j=1}^{4}\alpha_{j}}$. Also, let $f_{j}\in \alpha_{j}$ for $j=1,\dots,4$, and let $\gamma\subset \beta$. Let 
$$\lambda_{1}=\alpha_{2}\cup\alpha_{3}\cup(\alpha_{4}-f_{4})\cup\gamma,$$
$$\lambda_{2}=\{f_{1}\}\cup \gamma,$$
$$\lambda_{3}=\alpha_{1}\cup (\alpha_{2}-f_{2})\cup\alpha_{4}\cup\gamma,$$
$$\lambda_{4}=\{f_{3}\}\cup\gamma,$$
and define the four $d$-dimensional subcubes $S_{j}=\{\lambda_{j}\cup x:x\subset \alpha_{j}\}$, $j\in [4]$.
% We define the four $d$-dimensional subcubes $S_{1},\dots,S_{4}$ as follows:
%$$S_{1}=\{x\cup\alpha_{2}\cup\alpha_{3}\cup(\alpha_{4}-f_{4})\cup\gamma : x\in \alpha_{1}\},$$
%$$S_{2}=\{\{f_{1}\}\cup x\cup\gamma :x\in \alpha_{2}\},$$
%$$S_{3}=\{\alpha_{1}\cup (\alpha_{2}-f_{2})\cup x\cup\alpha_{4}\cup\gamma :x\in \alpha_{3}\},$$
%$$S_{4}=\{\{f_{3}\}\cup x\cup\gamma:x\in \alpha_{4}\}.$$
The disjoint union $A=\bigcup_{j=1}^{4}S_{j}$ is a $d$-dimensional absorber. If $A$ is an absorber, we denote the corresponding parameters $\alpha_{j},\lambda_{j},f_{j},S_{j},\beta,\gamma$ by $\alpha^{A}_{j},\lambda_{j}^{A},f^{A}_{j},S^{A}_{j},\beta^{A},\gamma^{A}$ for $j=1,\dots,4$, respectively.

The sets $S_{1},...,S_{4}$ are designed to satisfy the following property. For every $x\in S_{2j-1}$, $x$ is $\subset$-larger than every element of $S_{2j}$ if $f_{2j-1}\in x$, otherwise, $x$ is incomparable to every element of $S_{2j}$  (indices are meant modulo $4$). Similarly, for every $x\in S_{2j}$, $x$ is $\subset$-smaller than every element of $S_{2j+1}$ if $f_{2j}\not\in x$, otherwise, $x$ is incomparable to every element of $S_{2j+1}$. This property is useful for the following reason. If $Q$ is a copy of $P$ in $S_{2j-1}$ such that the minimal element $x$ of $Q$ contains $f_{2j-1}$ as a special element, then for every $x'\in S_{2j}$, $(Q-x)\cup\{x'\}$ is also a copy of $P$. Roughly saying, this means that we can use minimal elements of copies of $P$ in $S_{2j-1}$ with special element $f_{2j-1}$ to fill holes in $S_{2j}$. Similarly, we can use maximal elements of copies of $P$  in $S_{2j}$ with special element $f_{2j}$ to fill arbitrary holes in $S_{2j+1}$. 

In our next claim, we show that $A$ has good absorption properties. That is, we prove that if $A$ is an absorber and $R\subset A$ is a small family, then $A-R$ can be almost partitioned into copies of $P$. In the proof of this claim, we use the following theorem of Methuku and P\'{a}lv\"{o}lgyi \cite{MP}. 

\begin{theorem}[\cite{MP}]\label{La}
Let $P$ be a poset. Then there exists a constant $C_{P}$ such that for every positive integer $n$, if the family $\mathcal{F}\subset 2^{[n]}$	satisfies $|\mathcal{F}|\geq C_{P}2^{n}/\sqrt{n}$, then $\mathcal{F}$ contains a copy of $P$.
\end{theorem}

We do not use the full strength of this theorem, the only consequence of the theorem we need is that any $P$-packing of $A-R$ can be greedily extended to a $P$-packing that covers $(1-o(1))$ proportion of $A-R$.

\begin{claim}\label{absorber}
	Let $r$ and $d$ be positive integers such that $d\geq 4C_{P}^{2}2^{2|P|+r} $, and let $A$ be a $d$-dimensional absorber. Let $R\subset A$ such that $|R|\leq r$. Then $A-R$ has a $P$-packing that covers all but at most $|P|-1$ elements. In particular, if $|P|$ divides $|A-R|$, then $A-R$ has a $P$-partition.
\end{claim}

\begin{proof}
For simplicity, let $\alpha_{j},f_{j},\dots$ denote the parameters $\alpha^{A}_{j},f^{A}_{j},\dots$. First, we shall find a dense $P$-packing in $S_{j}$ in which $f_{j}$ is a special element for every copy of $P$. Suppose that $j$ is odd, the other case can be handled similarly. %Let $b\in P$ be a minimal element and let $\pi_{j}:P\rightarrow \alpha_{j}$ be an injection for which $\pi_{j}(b)=f_{j}$. For any $z\subset \alpha_{j}\setminus \pi(P)$,  consider the embedding $\phi_{z}:P\rightarrow 2^{[n]}$ defined by 
%$$\phi_{z}(p)=\{\pi(q):q\leq_{P} p\}\cup z\cup \lambda_{j}.$$
%Then $\phi_{z}(P)$ is a copy of $P$, in which $f_{j}$ is a special element for $\phi_{z}(b)$, and 
By Claim \ref{manyembedding}, $S_{j}$ contains $2^{d-|P|}$ disjoint copies $Q$ of $P$ for which $f_{j}$ is special for some minimal element of $Q$. Among these copies, let $\mathcal{P}_{j}$ be the collection of those that are disjoint from $R$. As at most $|R|$ of them can intersect $R$, we have $|\mathcal{P}_{j}|\geq 2^{d-|P|}-|R|$.

%$$\mathcal{P}_{j}=\{\phi_{z}(P):z\subset \alpha_{j}\setminus \pi(P), \phi_{z}(P)\cap R=\emptyset\}$$
% is a $P$-packing in $S_{j}\setminus R$ with at least $2^{d-|P|}-|R|$ copies of $P$.
  We slightly modify the $P$-packings $\mathcal{P}_{1},\mathcal{P}_{2},\mathcal{P}_{3}$ to ensure that the number of uncovered elements of $S_{j}$ is divisible by $|P|$ for $j=1,2,3$. We do this in the following way. First, we replace at most $|P|-1$ elements $Q\in \mathcal{P}_{1}$ with $Q\setminus \{x\}\cup \{y\}$, where $x$ is the minimal element of $Q$ for which $f_{1}$ is special, and $y\in S_{2}\setminus R$ is any uncovered point. Then, we replace at most $|P|-1$ elements $Q\in \mathcal{P}_{2}$ with $Q\setminus \{x\}\cup \{y\}$, where $x$ is a maximal element of $Q$ with special element $f_{2}$, and $y\in S_{3}\setminus R$ is any uncovered element. Finally, we replace at most $|P|-1$ elements $Q\in \mathcal{P}_{3}$ with $Q\setminus \{x\}\cup \{y\}$, where $x$ is a minimal element of $Q$ for which $f_{3}$ is special, and $y\in S_{4}\setminus R$ is any uncovered element. Let the resulting $P$-packings be $\PP'_{1},\dots,\PP'_{4}.$ (Here, $\PP_{4}=\PP'_{4}$.) Let $M_{j}$ denote the family of minimal (or maximal, if $j$ is even) elements of the copies of $P$ in $\PP'_{j}$ for which $f_{j}$ is special, and are contained in $S_{j}$. Then $|M_{j}|\geq 2^{d-|P|}-|R|-|P|$.
 
 Now, applying Theorem \ref{La}, we can extend $\PP'_{j}$ to a $P$-packing $\PP''_{j}$ that covers all but at most $C_{P}2^{d}/\sqrt{d}$ elements of $S_{j}\setminus R$, let $T_{j}$ denote the uncovered elements. Furthermore, we can find a $P$-packing $\mathcal{Q}_{j}$ in $M_{j}$ that covers all but at most $C_{P}2^{d}/\sqrt{d}$ elements of $M_{j}$. We chose $d$ such that $|M_{j}|-C_{P}2^{d}/\sqrt{d}> |T_{j'}|$ holds for every $j,j'\in [4]$. 
 
 We define our final packing $\PP$ as follows. For $j=1,\dots,4$, pick an arbitrary  sub-collection $\mathcal{Q}'_{j}$ of  $\mathcal{Q}_{j}$ of size $\lfloor |T_{j+1}|/|P|\rfloor$ (indices are meant modulo $4$), and let $M'_{j}$ be the elements of $M_{j}$ covered by $\mathcal{Q}'_{j}$. Let ${\pi_{j}: M'_{j}\rightarrow T_{j+1}}$ be any injection. Note that if $j\in \{1,2,4\}$, then $\pi_{j}$ is a bijection as $|P|$ divides $|T_{j+1}|$. Every $x\in M'_{j}$ is a minimal (or maximal, if $j$ is even) element of some $Q\in\PP''_{j}$ for which $f_{j}$ is a special element. Replace each such $Q$ with $Q\setminus \{x\}\cup\{\pi_{j}(x)\}$ and let $\PP^{*}_{j}$ be the resulting packing. Our final packing is defined as $\PP=\bigcup_{j=1}^{4}\PP_{i}^{*}\cup\bigcup_{j=1}^{4}\mathcal{Q}'_{j}$. Then $\PP$ is a $P$-packing of $A$ that covers every element of $S_{1}\setminus R,S_{2}\setminus R_{2},S_{3}\setminus R$, and all but at most $|P|-1$ elements of $S_{4}\setminus R$.

\end{proof}

We shall use that not only $A$ has good absorption properties, but $2^{[s]}\times A$ as well for any positive integer $s$. To prove this, we need the following well known result about Hamilton-cycles in $2^{[s]}$.

\begin{lemma}\label{Hamiltonpath}
  Let $s$ be a positive integer. There exists an enumeration $x_{1},\dots,x_{2^{s}}$ of the elements of $2^{[s]}$ such that $x_{i}$ and $x_{i+1}$ are comparable for $i=1,\dots,2^{s}-1$.
\end{lemma}

In fact, there exists an enumeration $x_{1},\dots,x_{2^{s}}$, where $|x_{i}\Delta x_{i+1}|=1$, see the Gray code \cite{Gray}, for example. However, the weaker statement of Lemma \ref{Hamiltonpath} is already enough for our purposes.

\begin{claim}\label{absorber_product}
Let $r,d,s$ be positive integers such that $d\geq 4C_{P}^{2}2^{4|P|+r} $, and let $A$ be a $d$-dimensional absorber. For each $x\in 2^{[s]}$, let $R_{x}\subset A$ such that $|R_{x}|\leq r$ and let $R=\bigcup_{x\in 2^{[s]}}\{x\}\times R_{x}$. Then  $(2^{[s]}\times A)-R$ has a $P$-packing that covers all but at most $|P|-1$ elements. In particular, if $|P|$ divides $2^{s}|A|-|R|$, then $(2^{[s]}\times A)- R$ has a $P$-partition.	
\end{claim}

\begin{proof}
  Let $t=2^{s}$ and let $x_{1},\dots,x_{t}$ be an enumeration of the elements of $2^{[s]}$ such that $x_{i}$ and $x_{i+1}$ are comparable. We define $P_{1},\dots,P_{t-1}$ such that $P_{i}$ is a copy of $P$, $P_{i}\subset (\{x_{i},x_{i+1}\}\times A)\setminus R$, and $|(\{x_{i}\}\times A)\setminus (R_{x_{i}}\cup P_{i}\cup P_{i-1})|$	is divisible by $|P|$ for $i=1,\dots,t-1$ (here, $P_{0}$ is assumed to be $\emptyset$). If we can find such $P_{1},\dots, P_{t-1}$, then by Claim \ref{absorber}, we can find a $P$-partition $\mathcal{P}_{i}$ of  $(\{x_{i}\}\times A)\setminus (R_{x_{i}}\cup P_{i}\cup P_{i-1})$ for $i=1,\dots, t-1$, and a $P$-packing $\mathcal{P}_{t}$ that covers all but at most $|P|-1$ elements of $(\{x_{t}\}\times A)\setminus (R_{x_{t}}\cup P_{t-1})$. But then 
  $$\left(\bigcup_{i=1}^{t}\mathcal{P}_{i}\right)\cup\{P_{1},\dots,P_{t-1}\}$$
  is a $P$-packing of $(2^{[s]}\times A)-R$ that covers all but at most $|P|-1$ elements.
  
  Hence, our task is reduced to finding suitable $P_{1},\dots,P_{t-1}$. We define $P_{1},\dots,P_{t-1}$ one by one, that is, if $P_{1},\dots,P_{j-1}$ is already defined, we define $P_{j}$ as follows. Let $Q$ be an arbitrary copy of $P$ in $A$ that does not contain an element of $R_{x_{j}}\cup R_{x_{j+1}}$, and $\{x_{j}\}\times Q$ is disjoint from $P_{j-1}$. There exists such a copy $Q$ of $P$, as $Q$ is only restricted to not contain at most $2r+|P|$ elements of $A$, so we can apply Theorem \ref{La} to any of the four $d$-dimensional subcubes forming $\{x_{j}\}\times A$, for example.
  
   Upon division by $|P|$, let $q\in \{0,\dots,p-1\}$ be the remainder of the size of ${(\{x_{j}\}\times A)\setminus (R\cup P_{j-1})}$. If $x_{j}\subset x_{j+1}$, then let $Q_{0}$ be a $q$ element downset of $Q$, and if $x_{j+1}\subset x_{j}$, then let $Q_{0}$ be a $q$ element upset of $Q$. Setting $$P_{j}=(\{x_{j}\}\times Q_{0})\cup(\{x_{j+1}\}\times (Q\setminus Q_{0})),$$ it can be easily checked that $P_{j}$ is also a copy of $P$, and $P_{1},\dots,P_{t-1}$ satisfy our desired conditions. 
\end{proof}

\subsection{Finding many disjoint absorbers}

Now we show that if $n$ is sufficiently large, then we can select a collection of absorbers in $T(n)$ such that every element of $T(n)$ can be  covered by a copy of $P$ which uses elements of the absorbers. To this end, if $F\subset 2^{[n]}$ and $x\in 2^{[n]}$, say that \emph{$F$ completes $x$} if there exist at least $|P|$ copies $Q\subset F\cup \{x\}$ of  $P$ such that $x\in Q$, and any two of these copies intersect only in $x$.

\begin{claim}\label{absorber_packing}
Let $d$ and $n$ be positive integers such that $d\geq 2|P|$ and $n\geq 10^{3}d\log d$. Then there exists a collection $\mathcal{A}$ of  $d$-dimensional absorbers in $2^{[n]}$ with the following properties:

\begin{itemize}
 \item [(1)] the elements of $\mathcal{A}$ are pairwise disjoint, 
 \item [(2)] no absorber in $\mathcal{A}$ contains a set with $1$ or $n-1$ elements, 
 \item [(3)] for each $x\in T(n)$ there exists $A\in \mathcal{A}$ such that $A$ completes $x$.
\end{itemize}
\end{claim}

\begin{proof}
	We show that a random collection of absorbers almost has the desired properties with high positive probability, and then we modify this family.
	
	Let $\mathcal{F}$ be the collection of all $d$-dimensional absorbers of $2^{[n]}$. Let $q=2^{-n/4}$ and pick every element of $\mathcal{F}$ with probability $q$. Let $\mathcal{A}_{0}$ be the family of picked absorbers. Let $G$ be the graph on $\mathcal{A}_{0}$ in which two elements are connected by an edge if they intersect.
	
	 Let $\mathcal{E}_{1}$ be the event that the maximum degree of $G$ is at least $k=5$. For a fixed absorber $A\in \mathcal{F}$, the number of absorbers $B\in\mathcal{F}$ that intersect $A$ is at most $2^{4d}d^{4}\binom{n}{d}^{4}\leq (2n)^{4d}$. This is true because there are less than $d^{4}\binom{n}{d}^{4}$ ways to choose the parameters $\alpha_{1}^{B},\dots,\alpha_{4}^{B}$ and $f_{1},\dots,f_{4}$, and if these parameters are fixed and $A\cap B\neq \emptyset$, then there are at most $2^{4d}$ choices for $\gamma^{B}$, as $\gamma^{A}$ and $\gamma^{B}$ can only differ on the set $\bigcup_{j=1}^{4} \alpha_{j}^{A}\cup \bigcup_{j=1}^{4} \alpha_{j}^{B}$. Hence, the probability that $A\in \mathcal{A}_{0}$ has degree at least $k$ is at most $\binom{(2n)^{4d}}{k}q^{k}$. But then, by the union bound $$\PV(\mathcal{E}_{1})\leq 2^{n}\binom{(2n)^{4d}}{k}q^{k}\leq 2^{n}(2n)^{4dk}q^{k}=2^{-n/4}(2n)^{20d}\leq 1/4.$$
	 
	 Let $\mathcal{E}_{2}$ be the event that there exists an absorber in $\mathcal{A}_{0}$ that contains a set with $1$ or $n-1$ elements. Let $u\in [n]$, then the absorber $A$ contains the set $\{u\}$ if and only if $u\in \alpha_{1}^{A}$, $f_{1}^{A}=u$ and $\gamma^{A}=\emptyset$, or $u\in \alpha_{3}^{A}$, $f_{3}^{A}=u$ and $\gamma^{A}=\emptyset$. Hence the number of absorbers in $\mathcal{F}$ containing $\{u\}$ is less than $\binom{n-1}{d-1}d^{3}\binom{n}{d}^{3}\leq n^{4d}$. Therefore, the probability that an absorber in $\mathcal{A}_{0}$ contains $\{u\}$ is at most $qn^{4d}$. Similarly, the probability that an absorber in $\mathcal{A}_{0}$ contains $[n]-u$ is at most $qn^{4d}$. Thus, the union bound gives that $\PV(\mathcal{E}_{2})\leq 2n\cdot qn^{4d}< 1/4$.
	 
	 Let $l=nk2^{k}=160n$ and let $\mathcal{E}_{3}$ be the event that for some $x\in T(n)$ there are at most $l-1$ absorbers $A\in \mathcal{A}_{0}$ that complete $x$. We shall bound the probability that for a given $x$ there are at most $l-1$ absorbers that complete $x$. Let us assume that $|x|\leq n/2$, the other case follows similarly. If the absorber $A$ satisfies that
	 \begin{enumerate}
	 	\item $|\alpha^{A}_{1}\cap x|=1$,
	 	\item $f_{4}^{A}\not\in x$,
	 	\item $x\cap \beta^{A}\subset \gamma^{A}$,
	 \end{enumerate}   
     then $S^{A}_{1}$ completes $x$. Indeed, let $s$ be the single base element in $\alpha^{A}_{1}\cap x$, then  by Claim \ref{manyembedding}, there exist at least $2^{d-|P|}> |P|$ disjoint copies of $P$ in $S^{A}_{1}$ for which the base element $s$ is special for some minimal element. Let $|P|$ of these copies be $Q_{1},\dots, Q_{|P|}$, and let $y_{l}\in Q_{l}$ be minimal such that $s$ is special for $y_{l}$. Then $x\subset y_{l}$ and $s\in x$, so $Q_{l}'=(Q_{l}-y_{l})\cup \{x\}$ is a copy of $P$ in $S_{1}^{A}\cup\{x\}$ by Claim \ref{special}. Therefore, $S_{1}^{A}$ completes $x$, so in particular, $A$ completes $x$.  
     
     The number of absorbers satisfying 1.,2., and 3. is at least $2^{n-|x|-4d}\geq 2^{n/2-4d}=N$, because after fixing the parameters $\alpha^{A}_{1},\dots,\alpha^{A}_{4},f^{A}_{1},\dots,f^{A}_{4}$ arbitrarily satisfying 1. and 2., we have at least $2^{n-|x|-4d}$ choices of $\gamma^{A}$ satisfying 3. Hence, the probability that there are no $l$ absorbers that complete $x$ is at most 
     $$\sum_{i=0}^{l-1}\binom{N}{i}(1-q)^{N-i}q^{i}< lN^{l}(1-q)^{N-l}< e^{l+l\log N-q(N-l)}<e^{300n^{2}-2^{n/4}}< \frac{1}{4\cdot 2^{n}}.$$
     By the union bound, we get $\PV(\mathcal{E}_{3})< 1/4$. 
     
     Hence, $\PV(\mathcal{E}_{1}\cup \mathcal{E}_{2}\cup \mathcal{E}_{3})\leq 3/4$, which means that there exists $\mathcal{A}_{0}$ for which $G$ has maximum degree at most $k-1$, no absorber in $\mathcal{A}_{0}$ contains a set with $1$ or $n-1$ elements, and for each $x\in T(n)$ there are at least $l$ absorbers in $\mathcal{A}_{0}$ that are good for $x$.
     
     Now we are in a position to define $\mathcal{A}$ satisfying (1), (2), and (3). For each edge of $G$, remove one of its endpoints from $\mathcal{A}_{0}$ with equal probability, and let $\mathcal{A}$ be the family of the remaining absorbers. Clearly, any two elements of $\mathcal{A}$ are disjoint, so (1) is satisfied. Also, (2) is satisfied as (2) holds in $\mathcal{A}_{0}$ as well. We finish the proof by showing that (3) holds with positive probability.
     
     The probability that $A\in \mathcal{A}_{0}$ survives in $\mathcal{A}$ is at least $1/2^{k}$ as $A$ is contained in at most $k$ edges of $G$. Also, if $\mathcal{I}\subset \mathcal{A}_{0}$ is an independent set in $G$, then the events $\mathcal{E}_{A}=\{A\in\mathcal{A}\}$ are pairwise independent for $A\in\mathcal{I}$. Let $x\in T(n)$ and let $\mathcal{B}_{x}\subset\mathcal{A}_{0}$ be the set of absorbers that complete $x$. As the maximum degree of $G$ is at most $k-1$ and $|\mathcal{B}_{x}|\geq l$, there exists an independent set $\mathcal{I}_{x}$ in $G[\mathcal{B}_{x}]$ of size at least $l/k$. The probability that no element of $\mathcal{I}_{x}$ survives in $A$ is at most $$(1-1/2^{k})^{l/k}\leq e^{-l/k2^{k}}<1/2^{n}.$$
      Thus, by the union bound, the probability that for each $x\in T(n)$ there exists an absorber in $\mathcal{A}$ that completes $x$ is at least $2/2^{n}$, so there is a choice for $\mathcal{A}$ that satisfies (1),(2) and (3).
\end{proof}

\subsection{Constructing a dense collection of absorbers}\label{sect:denseabsorber}

Let $d=\lceil 4C_{P}2^{6|P|}\rceil$ and $n_{1}=\lceil 10^{3}d\log d\rceil$. Then by Claim \ref{absorber_product}, if $m$ is a positive integer, $A$ is a $d$-dimensional absorber, and $R\subset 2^{[m]}\times A$ such that $|R\cap (\{x\}\times A)|\leq 2|P|$ for every $x\in 2^{[m]}$, then $(2^{[m]}\times A)-R$ has a $P$-packing that covers all but at most $|P|-1$ elements. Also, by Claim \ref{absorber_packing}, there exists a collection $\mathcal{A}$ of absorbers in $2^{[n_{1}]}$ that satisfy the properties (1), (2) and (3).

Let $k$ be  positive integers and let $n_{2}=n_{1}k$. We shall view $2^{[n_{2}]}$ as the cartesian power $B(k)=(2^{[n_{1}]})^{k}$, that is, the elements of $B(k)$ have the form  $(x_{1},\dots,x_{k})$, where $x_{1},\dots,x_{k}\in 2^{[n_{1}]}$. If $x\in B(k)$, $x(i)$ denotes the $i$-th coordinate of $x$. Let $B^{-}(k)=B(k)-\{\min B(k),\max B(k)\}$.

Given a collection $\mathcal{F}$ of families of $2^{[n_{1}]}$, let $\mathcal{F}^{(k)}$ be the collection of subsets of $B(k)$ of the form 
$$\{x_{1}\}\times\dots\times \{x_{m-1}\}\times F\times\{x_{m+1}\}\times\dots\times\{x_{k}\},$$
where $m\in [k]$, $x_{1},\dots,x_{m-1},x_{m+1},\dots, x_{k}\in 2^{[n_{1}]}$, $F\in\mathcal{F}$, and $x_{i}$ is not contained in any member of $\mathcal{F}$ for $i<m$. That is, if $F'\in \mathcal{F}^{(k)}$, then the projection of $F'$ to exactly one of the coordinates is a member of $\mathcal{F}$, while its projection to each of the other coordinates is a single element of $2^{[n_{1}]}$. We shall use the following simple properties of $\mathcal{F}^{(k)}$ that we state without proof.

\begin{claim}\label{power_properties}
	Let $\mathcal{F}$ be a collection of pairwise disjoint families of $2^{[n_{1}]}$.
	
	(1) The members of $\mathcal{F}^{(k)}$ are pairwise disjoint.
	
	(2) $x\in B(k)$ is covered by some member of $\mathcal{F}^{(k)}$ if and only if at least one coordinate of $x$ is covered by some member of $\mathcal{F}$.
\end{claim}

Say that  $x\in B(k)$ is \emph{problematic} if the number of indices $i\in [k]$ for which $x(i)$ is neither $\emptyset$ or $[n_{0}]$ is at most $2^{n_{1}}$. Say that $x$ is \emph{ordinary}, if it is not problematic. Let $\mathcal{B}$ be the collection of absorbers in $\mathcal{A}^{(k)}$ which do not contain a problematic element. Note that if $A\in \mathcal{A}^{(k)}-\mathcal{B}$, then every element of $A$ is problematic.

\begin{claim}\label{absorber_matching}
	Let $O$ be the family of ordinary points $x\in B(k)$ that are not covered by any member of $\mathcal{B}$. There exists a complete matching from $O$ to $\mathcal{B}$ in which each $x\in O$ is matched to $A\in \mathcal{B}$, where $A$ completes $x$.
\end{claim}

\begin{proof}
 Consider the following bipartite graph $G$ between $O$ and $\mathcal{B}$. Let $A=\{x_{1}\}\times\dots\times \{x_{m-1}\}\times A_{0}\times\{x_{m+1}\}\times\dots\times\{x_{k}\}$ be a member of $\mathcal{B}$, where $A_{0}\in \mathcal{A}$. Join $A$ and $x\in O$ by an edge if $A_{0}$ completes $x(m)$ and $x(i)=x_{i}$ for $i\in [k]-m$. 
 
 With the help of Hall's theorem \cite{H}, we show that there is a complete matching from $O$ to $\mathcal{B}$ in $G$. The degree of every member of $\mathcal{B}$ in $G$ is at most $2^{n_{1}}$, while the degree of every $x\in O$ is at least the number of coordinates of $x$ which do not equal to $\emptyset$ or $[n_{1}]$. The latter is true because for every $x_{0}\in T(n_{1})$ there is an absorber $A_{0}\in \mathcal{A}$ that completes $x_{0}$, so for every coordinate $x(m)$ of $x$ that is not $\emptyset$ or $[n_{1}]$, there is an absorber $A_{0}$ such that $A=\{x(1)\}\times\dots\times \{x(m-1)\}\times A_{0}\times\{x(m+1)\}\times\dots\times\{x(k)\}$ is joined to $x$ by an edge. Clearly, this $A$ is an element of $\mathcal{B}$ as $x(1),\dots,x(k)$ are not covered by any element of $\mathcal{A}$, see (2) in Claim \ref{power_properties}. But as the elements of $O$ are ordinary, this implies that the degree of every $x\in O$ is at least $2^{n_{1}}$.
 
 Therefore, Hall's condition holds. Indeed, let $U\subset O$ and let $V$ be the set of neighbors of $U$ in $G$. By double counting the number of edges $e$ between $U$ and $V$, we arrive to the inequality $|U|2^{n_{1}}\leq e\leq |V|2^{n_{1}}$, which gives $|U|\leq |V|$.
\end{proof}

Now let us deal with the problematic elements of $B(k)$. Say that $x\in B(k)$ is \emph{restricted} if for every $i\in [k]$, we have $|x(i)|\in \{0,1,n_{1}-1,n_{1}\}$, and $x$ is not problematic. Note that by property (2) in Claim \ref{absorber_packing}, and by Claim \ref{power_properties}, no restricted element is covered by any member of $\mathcal{B}$. We shall use restricted elements of $B(k)$ to cover the problematic elements. Let $PR$ denote the family of elements in $B(k)$ that are either problematic or restricted.

\begin{claim}\label{problematic}
	Let $k\geq 100n_{1}2^{n_{1}}$. Then there exists a $P$-packing $\mathcal{P}$ in $PR$ that covers every problematic element of $B(k)^{-}$, and each member of $\mathcal{P}$ contains exactly $1$ problematic element.
\end{claim}

\begin{proof}
Let $N=2^{n_{1}}$ and let $M$ be the number of problematic elements in $B(k)$. We have $M<3^{k}N^{N}$, because for each coordinate of $x\in B(k)$, there are three choices: either $x(i)=\emptyset$, $x(i)=[n_{1}]$ or $x(i)\in T(n_{1})$. If $x$ is problematic, there are at most $N$ indices $i\in [k]$ such that $x(i)\in T(n_{1})$, so there are less than $N^{N}$ choices for the values of these coordinates.

We show that for each problematic element $x$, there is a collection $\mathcal{P}_{x}$ of at least $n_{1}^{(k-N)/2-|P|}$ copies of $P$ in $B(k)$ such that for each $Q\in \mathcal{P}_{x}$, we have $x\in Q$, every element of $Q-x$ is restricted, and any two members of $\mathcal{P}_{x}$ intersect only in $x$. If we are able to prove this, we are done. Indeed, we chose $k$ such that the inequality $$|\mathcal{P}_{x}|\geq M|P|$$ holds, which means that for each problematic $x\in B(k)$, we can greedily pick $Q_{x}\in\mathcal{P}_{x}$ such that $\{Q_{x}:x\mbox{ is problematic}\}$ is a $P$-packing with the desired properties.

To this end, let $x\in B(k)^{-}$ be problematic. Let $\alpha\subset [k]$ be the set of indices $i$ such that $x(i)=\emptyset$, let $\beta\subset [k]$ be the indices $i$ such that $x(i)=[n_{1}]$, and let $\gamma=[k]-(\alpha\cup\beta)$. Suppose that $|\alpha|\geq |\beta|$, the other case can be handled in a similar manner. As $|\gamma|\geq N$, we have $|\alpha|\geq (k-N)/2$. 

Without loss of generality, let $\alpha=\{1,\dots,|\alpha|\}$. Also, as $x\neq \min B(k)$, there exist $s\in\beta\cup\gamma$ and $t\in x(s)$. Let $P'$ be a copy of $P$ in $2^{[|P|]}$ in which the single element set $p_{0}=\{|P|\}$ is minimal; by Claim \ref{embedding}, there exists such a copy.  For each $u_{i}\in [n_{1}]$, where $i\in \{|P|,|P|+1,\dots,|\alpha|\}$, we define the family $Q=Q_{u_{|P|},\dots,u_{|\alpha|}}$ with the help of $P'$. For each $p\in P'$, let $z_{p}\in B(k)$ be the element, whose $j$-th coordinate is defined as follows.
\begin{itemize}
\item  If $j\in \{1,\dots,|P|-1\}$, then $z_{p}(j)=\{1\}$ if $j\in p$, otherwise $z_{p}(j)=\emptyset$;
\item  if $j\in \{|P|,\dots,|\alpha|\}$, then $z_{p}(j)=\{u_{j}\}$;
\item  if $j=s$, then $z_{p}(j)=[n_{1}]$ if $|P|\in p$, otherwise $z_{p}(j)=[n_{1}]-t$;
\item  if $j\in \beta\cup\gamma-s$, then $z_{p}(j)=[n_{1}]$.
\end{itemize}
Set $Q=\{z_{p}:p\in P'\}$. Then $Q$ is a copy of $P$. Indeed, $Q$ is constant outside of the $|P|$ coordinates $\{1,\dots,|P|-1,s\}$; also, in these coordinates $Q$ takes two different values, so $Q$ lives in the $|P|$-dimensional subcube of $B(k)$ determined by these two values and $|P|$ coordinates, where it is designed to be isomorphic to $P'$. Moreover, every element of $Q$ is restricted. Indeed, every coordinate of $z\in Q$ has either $0,1,n_{1}-1$ or $n_{1}$ elements, but $z$ is not problematic as it has at least $|\alpha|-|P|+1>N$ coordinates of size $1$. Finally, $$\{Q_{u_{|P|},\dots,u_{|\alpha|}}:u_{|P|},\dots,u_{|\alpha|}\in [n_{1}]\}$$ is a packing. Indeed, if $(u_{|P|},\dots,u_{|\alpha|})\neq (u'_{|P|},\dots,u'_{|\alpha|})$, then every element of $Q_{u_{|P|},\dots,u_{|\alpha|}}$ differs in at least one of the coordinates indexed by $\{|P|,\dots,|\alpha|\}$ from every element of $Q_{u'_{|P|},\dots,u'_{|\alpha|}}$.

Now we would like to replace an element of $Q$ with $x$. As $\{|P|\}$ $z=z_{p_{0}}\in Q$. Then $z$ is a minimal element of $Q$ which satisfies $z(j)=\emptyset$ for $j\in\{1,\dots,|P|\}$ and $z(s)=[n_{1}]$. Also, for every $z'\in Q$, we have $z\not<z'$ if and only if $z'(s)=[n_{1}]-t$. But $x<z$ and $t\in x(s)$. Hence, $Q'=(Q-z)\cup\{x\}$ is also a copy of $P$. 

To conclude our proof, note that 
$$\mathcal{P}_{x}=\{Q'_{u_{|P|},\dots,u_{|\alpha|}}:u_{|P|},\dots,u_{|\alpha|}\in [n_{1}]\}$$
is a collection of $n_{1}^{|\alpha|-|P|+1}$ copies of $P$ containing $x$, every element of $Q'-x$ is restricted for $Q'\in \mathcal{P}_{x}$, and any two members of $\mathcal{P}_{x}$ intersect only in $x$.
\end{proof}

\subsection{Divisibility conditions}\label{sect:modp}
As before, $PR$ denotes the family of elements in $B(k)$ that are problematic or restricted. Also, we remind the reader that $T(n)=2^{[n]}-\{\emptyset,[n]\}$.

If $F$ is a subset of some ground set $X$, let $\chi_{F}:X\rightarrow \mathbb{Z}$ be the characteristic function of $F$, that is, $\chi_{F}(x)=1$ if $x\in F$, and $\chi_{F}(x)=0$ otherwise.

 Let $m$ be a positive integer. Say that a function $f:T(m)\rightarrow \mathbb{Z}_{|P|}$ is \emph{realizable} if there exists $P_{1},\dots,P_{s}\subset T(m)$  such that $P_{i}$ is a copy of $P$ for $i\in [s]$ and $f\equiv\sum_{i=1}^{s}\chi_{P_{i}}\Mod{|P|}.$ Also, let $f:B(k)\rightarrow \mathbb{Z}_{|P|}$ be \emph{strongly realizable} if there exists $P_{1},\dots,P_{s}\subset B(k)$ such that $P_{i}$ is a copy of $P$ disjoint from PR for $i\in [s]$, and $f\equiv\sum_{i=1}^{s}\chi_{P_{i}}\Mod{|P|}.$
   
Clearly, if $c\in \mathbb{Z}_{|P|}$ and $f$ and $g$ are (strongly) realizable, then $cf$ and $f+g$ are also (strongly) realizable. We aim to prove the following lemma in this subsection. We follow a similar line of proof as in Lemma 4' in \cite{GLT2}, however, our proof is more technical.

\begin{lemma}\label{modp}
	Let $f:B(k)\rightarrow \mathbb{Z}_{|P|}$ be a function such that $f(x)=0$ if $x\in PR$, and $\sum_{x\in B(k)}f(x)=0$. Then $f$ is strongly realizable. 
\end{lemma}

\begin{proof}

For $a,b\in B(k)$, let $f_{a,b}=\chi_{\{a\}}-\chi_{\{b\}}$. To prove our lemma, it is enough to show that $f_{a,b}$ is strongly realizable for every $a,b\in B(k)\setminus PR$, as every function satisfying the conditions of Lemma \ref{modp} is the sum of such functions $f_{a,b}$. Consider the graph $G$ on $B(k)\setminus PR$ in which $a$ and $b$ are joined by an edge if $f_{a,b}$ is strongly realizable. We wish to show that $G$ is the complete graph, but it is enough to show that $G$ is connected, because of the identity $f_{a,b}+f_{b,c}=f_{a,c}$. We show that $G$ is connected step by step. First, we consider realizable functions in $2^{[m]}$.

\begin{claim}\label{2elements}
	Let $m\geq 2|P|+2$ and $x,y\in T(m)$. Then $g_{x,y}=\chi_{\{x\}}-\chi_{\{y\}}$ is realizable. Also, if $|x|,|y|\not\in \{1,m-1\}$, then there exists $P_{1},\dots,P_{s}\subset T(m)$ such that $P_{i}$ is a copy of $P$ not containing a set with $1$ or $m-1$ elements for $i\in [s]$, and $g_{x,y}\equiv\sum_{i=1}^{s}\chi_{P_{s}}\Mod{|P|}.$
\end{claim}

\begin{proof}
	Say that a copy of $P$ is \emph{good} if it does not contain a set with $1$ or $m-1$ elements. Also, let $\{x,y\}\in T(m)^{(2)}$ be \emph{good} if there exists $P_{1},\dots,P_{s}\subset T(m)$ such that $P_{i}$ is a good copy of $P$ and $g_{x,y}\equiv\sum_{i=1}^{s}\chi_{P_{s}}\Mod{|P|}.$
	
	First, suppose that $x\subset y$ and $|y|\leq m-|P|-1$. Let $i\in x$ and let $\alpha\subset [m]$ be a set of $|P|-1$ elements such that $\alpha\cap y=\emptyset$. By Claim \ref{embedding}, the $|P|$-dimensional cube $2^{\alpha\cup\{i\}}$ contains a copy $Q$ of $P$ such that $\{i\}\in Q$ is minimal. Let 
	$$Q'=\{z\cup (y-i):z\in Q\},$$
    then $Q'$ is a copy of $P$ in which $y$ is minimal with special element $i$. But then $Q''=(Q'-y)\cup \{x\}$ is also a copy of $P$ and $$g_{x,y}\equiv\chi_{Q''}+(p-1)\chi_{Q'}\Mod{|P|}.$$
    Thus $g_{x,y}$ is realizable. Also, if $|x|\neq 1$, then $|y|\geq 3$, so $|z\cup (y-i)|\geq 2$ for any $z\in Q$. Therefore, $Q'$ does not contain a set with a $1$ element. Also, $|z\cup (y-i)|\leq |P|+m-|P|-2=m-2$, so $Q'$ does not contain a set with $m-1$ elements either. Therefore, $Q'$ and $Q''$ are good and $\{x,y\}$ is good. 
    
     Similarly, if $x\subset y$ and $|x|\geq |P|+1$, then $g_{x,y}$ is realizable. Also, if $|y|\neq m-1$, then $\{x,y\}$ is good.
     
      But then for every $x,y\in T(m)$ satisfying $x\subset y$, $g_{x,y}$ is realizable. Indeed, if $$|x|<|P|+1\leq m-|P|-1<|y|,$$ then there exists $z\in T(m)$ such that $|P|+1\leq |z|\leq m-|P|-1$ and $x\subset z\subset y$, but then $g_{x,z}$ and $g_{z,y}$ are realizable, and $g_{x,y}=g_{x,z}+g_{z,y}$. Also, $|z|\not\in \{1,m-1\}$, so if $|x|,|y|\not\in\{1,m-1\}$, then $\{x,z\}$ and $\{z,y\}$ are good, which gives that $\{x,y\}$ is also good.
    
    Now let $x,y\in T(m)$ arbitrary. If $z=|x\cap y|>1$, then $g_{x,z}$ and $g_{z,y}$ are realizable, so $g_{x,y}$ is also realizable. Also, $|z|\not\in\{1,m-1\}$, so if $|x|,|y|\not\in\{1,m-1\}$, then $\{x,y\}$ is good. We can argue similarly, if $|x\cup y|<m-1$. 
    
    The only case remaining if $|x\cap y|\leq 1$ and $|x\cup y|\geq m-1$. Without loss of generality, let $|x|\geq |y|$, then $|x|\geq m/2-1$ and $|y|\geq m/2+1$. In this case, we can find $z$ such that $z\not\in\{1,m-1\}$,  $|x\cap z|>1$ and $|y\cup z|<m-1$. But then $g_{x,z}$ and $g_{y,z}$ are realizable, so $g_{x,y}$ is realizable. Also, if $|x|,|y|\not\in\{1,m-1\}$, then $\{x,z\}$ and $\{z,y\}$ are good, so $\{x,y\}$ is good.
\end{proof}

\noindent
\textbf{Step 1.} Let $x,y\in B(k)-PR$ be two vectors that only differ in one coordinate, say in the $l$-th coordinate, and neither $x(l)$ or $y(l)$ is equal to $\emptyset$ or $[n_{1}]$. Then $f_{x,y}$ is strongly realizable, as we can apply the previous claim to the $n_{1}$-dimensional subcube of the form $$(x(1),\dots,x(l-1),2^{[n_{1}]},x(l+1),\dots,x(k)).$$ Indeed, if $P_{1},\dots,P_{s}$ are copies of $P$ in $2^{[n_{1}]}$ such that $$g_{x(l),y(l)}\equiv \sum_{i=1}^{s}\chi_{P_{i}}\Mod{|P|},$$
then 
$$f_{x,y}\equiv \sum_{i=1}^{s}\chi_{(x(1),\dots,x(l-1),P_{i},x(l+1),\dots,x(k))}\Mod{|P|}.$$
If either $|x(l)|\in \{1,n_{1}-1\}$ or $|y(l)|\in \{1,m-1\}$, then for any $z\in T(n_{1})$, the vector $(x(1),\dots,x(l-1),P_{i},x(l+1),\dots,x(k))$ is not in PR, so $(x(1),\dots,x(l-1),P_{i},x(l+1),\dots,x(k))$ is disjoint from PR. Also, if $|x(l)|,|y(l)|\not\in \{1,n_{1}-1\}$, then by the second part of Claim \ref{2elements}, we can find $P_{1},...,P_{s}$ such that $P_{i}$ does not contain a set with $1$ or $m-1$ elements, so $(x(1),\dots,x(l-1),P_{i},x(l+1),\dots,x(k))$ is disjoint from PR.

\bigskip
\noindent
\textbf{Step 2.} For $x\in B(k)$, let $\alpha(x)=(\beta(x),\gamma(x))$, where $\beta(x)$ is the set of indices $i\in [k]$ such that $x(i)=\emptyset$, and $\gamma(x)$ is the set of indices $i\in [k]$ such that $x(i)=[n_{1}]$.  For every $\alpha=(\beta,\gamma)\in 2^{[k]}\times 2^{[k]}$, let $B_{\alpha}=\{x\in B(k)-PR:\alpha(x)=\alpha\}$. Here, $B_{\alpha}$ is non-empty if and only if $\beta\cap\gamma=\emptyset$ and $|\beta|+|\gamma|< k-2^{n_{1}}$. By the previous observation, we get that if $x,y\in B_{\alpha}$, then $f_{x,y}$ is strongly realizable. Indeed, we can find a sequence $x=z_{0},z_{1},...,z_{s}=y$ such that $z_{i}$ and $z_{i+1}$ only differ in one coordinate $l$, where $l\not\in \beta(x)\cup\gamma(x)$, and then $f_{z_{i},z_{i+1}}$ is strongly realizable, so $f_{x,y}=\sum_{i=0}^{s}f_{z_{i},z_{i+1}}$ is also strongly realizable. Thus, the graph $G[B_{\alpha}]$ is connected.

\bigskip
\noindent
\textbf{Step 3.} Now fix $\alpha=(\beta,\gamma)$ such that $B_{\alpha}$ is non-empty, let $l\in [k]-(\beta\cup\gamma)$, and let $x\in B_{\alpha}$ be any element that satisfies $|x(l)|\leq n_{1}-|P|$. There exists a copy $Q$ of $P$ in $T(n_{1})$ such that $x(l)$ is minimal in $Q$. Let $y\in B(k)$ such that $x(j)=y(j)$ for every $j\in [k]-\beta$, and $|y(j')|=2$ for every $j'\in \beta$. Then $x\subset y$ and $y\in B_{(\emptyset,\gamma)}$. Let 
$$Q'=\{y(1)\}\times\dots\times\{y(l-1)\}\times Q\times \{y(j+1)\}\times\{y(k)\},$$
then $Q'$ is a copy of $P$ in which $y$ is a minimal element. But $Q''=(Q'-y)\cup x$ is a also a copy of $P$. As $$f_{x,y}\equiv\chi_{Q''}+(p-1)\chi_{Q'}\Mod{|P|},$$
we deduce that $f_{x,y}$ is strongly realizable. This means that the sets  $B_{(\beta,\gamma)}$ and $B_{(\emptyset,\gamma)}$ are connected by an edge in $G$. We can prove similarly that $B_{(\beta,\gamma)}$ and $B_{(\beta,\emptyset)}$ are also connected by an edge in $G$. But then $B_{(\emptyset,\emptyset)}$ is connected to every $B_{\alpha}$, so $G$ is truly connected.

\end{proof}

As there are only finitely many different functions $f:B(k)\rightarrow \mathbb{Z}_{|P|}$, there exists a positive integer $N$ such that if $f$ is strongly realizable, then there exist at most $N$ copies of $P$ in $B(k)\setminus PR$ whose characteristic functions add up to $f$ modulo $|P|$. With careful analysis of the proof, one can show that $N=O(k|P|^{2}2^{n_{2}})$. Indeed, every $g_{x,y}$ in Claim \ref{2elements} can be realized with at most $O(|P|)$ copies of $P$, and every $f_{a,b}$ can be realized with $O(k|P|)$ copies of $P$.  Also, every function $f$ satisfying the conditions of Lemma \ref{modp} is the sum of at most $|P||B(k)|=|P|2^{n_{2}}$ functions of the form $f_{a,b}$. However, we shall not use this quantitative bound on $N$.

\subsection{Adding more dimensions}\label{sect:correction}
 
 Let $n_{3}$ be a positive integer.  We shall view $2^{[n_{3}+n_{2}]}$ as the cartesian product $C=2^{[n_{3}]}\times B(k)$, that is, the elements of $C$ are $(x,y)$, where $x\in 2^{[n_{3}]}$ and $y\in B(k)$. The aim of this subsection is to prove the following lemma. 
 
 \begin{lemma}\label{minmaxlemma}
 	There exists a $P$-packing $\mathcal{P}$ in $C=2^{[n_{3}]}\times B(k)$ that covers every element of ${T(n_{3})\times \{\min B(k),\max B(k)\}}$, and if $(x,y)\in 2^{[n_{3}]}\times B(k)$ is covered by a member of $\mathcal{P}$, then $y$ is problematic.
 \end{lemma}

We prepare the proof of this lemma with the following two claims.
 
 \begin{claim}\label{smallsets}
 	Let $m>5|P|$ be an integer and let $0<c\leq 0.1$. There exists a $P$-packing that covers every $x\in T(m)$ satisfying $|x|\leq cm$ or $|x|\geq (1-c)m$. 
 \end{claim}

\begin{proof}
	Let $x\in T(m)$ such that $|x|\leq cm$. First, we show that there is a collection $\mathcal{P}_{x}$ of copies of $P$ such that $|\mathcal{P}_{x}|\geq 2^{(1-c)m-|P|}$, every member of $\mathcal{P}_{x}$ contains $x$, and any two members of $\mathcal{P}_{x}$ intersect only in $x$. Let $i\in x$ and let $\alpha\subset [m]-x$ be an arbitrary subset of size $|P|-1$. By Claim \ref{embedding}, there is a copy $Q$ of $P$ in $2^{\alpha\cup\{i\}}$ in which $\{i\}$ is a minimal element. For every $\alpha\subset [m]-(x\cup \alpha)$, let $$Q_{\alpha}=\{q\cup \alpha\cup (x-i):q\in Q\}.$$
	Then $(Q_{\alpha}-(x\cup \alpha))\cup \{x\}$ is also a copy of $P$ and the collection $\{Q_{\alpha}: \alpha\subset [m]-(x\cup \alpha)\}$ contains $2^{m-|x|-|P|+1}$ disjoint copies of $P$. Hence, setting 
	$$\mathcal{P}_{x}=\{(Q_{\alpha}-(x\cup \alpha))\cup \{x\}: \alpha\subset [m]\setminus (x\cup \alpha)\}$$
	suffices. If $|x|\geq (1-c)m$, we can also find a collection $\mathcal{P}_{x}$ with the same properties.
	
	But then we can construct our desired $P$-packing greedily. The number of elements $x\in T(m)$ with $|x|\leq cm$ or $|x|\geq (1-c)m$ is at most $2\sum_{l=0}^{cm}\binom{m}{l}<2\cdot2^{H(c)m}$, where $H(c)=-c\log_{2}c-(1-c)\log_{2}(1-c)$ is the binary entropy function. Pick copies of $P$ one-by-one to cover every $x\in T(m)$ satisfying $|x|\leq cm$ or $|x|\geq (1-c)m$. If there are $M$ copies of $P$ picked so far and we wish to cover $x$, then these $M$ copies can intersect at most $|P|M$ members of $\mathcal{P}_{x}$. Therefore, as long as $M|P|<|\mathcal{P}_{x}|$, we can find a copy of $P$ that covers $x$ and is disjoint from every previously picked copy of $P$.
	
	It only remains to verify that $M|P|<|\mathcal{P}_{x}|$, which follows from the inequality $$2\cdot 2^{H(c)m}<2^{(1-c)m-|P|}.$$ If $c\leq 0.1$, then $H(c)<0.5$, so the inequality is satisfied noting that ${m>5|P|}$. 
\end{proof}

\begin{claim}\label{smalldegreegraph}
	Let $m$ be a positive integer and let $0<c<1/2$. There exists a graph $G$ on $2^{[m]}$ such that the maximum degree of $G$ is at most $2\lceil 1/c \rceil$, and for every $x\in 2^{[m]}$ satisfying $cm\leq |x|\leq (1-c)m$, there exist $x_{1},x_{2}\in 2^{[m]}$ such that $x_{1}\subset x\subset x_{2}$, and $xx_{1}$, $xx_{2}$ are edges of $G$.   	
\end{claim}

\begin{proof}
	For $l=0,\dots,m$, let $L_{l}=\{x\in 2^{[m]}:|x|=l\}$. We show that if $cm\leq l\leq (1-c)m$, then the comparability graph between $L_{l}$ and $L_{l+1}$ contains a subgraph of minimum degree at least $1$ and maximum degree at most $\lceil 1/c\rceil$. Then we are done as we can take $G$ to be the union of these graph for $l=cm,\dots,(1-c)m$.
	
	We shall use the following simple consequence of Hall's theorem \cite{H}: if $H=(A,B;E)$ is a bipartite graph in which the degree of every vertex in $A$ is $d_{A}$, and the degree of every vertex in $B$ is $d_{B}$, then there exists a complete matching from $A$ to $B$ if $d_{A}\geq d_{B}$.
	
	Suppose that $cm\leq l<m/2$, the other case can be handled in a similar fashion. Let $B$ be the bipartite comparability graph between $L_{l}$ and $L_{l+1}$. As the degree of every vertex in $L_{l}$ is $n-l$, and the degree of every vertex in $L_{l+1}$ is $l+1$, where $n-l\geq l+1$, there is a complete matching $M_{1}$ from $L_{l}$ to $L_{l+1}$.
    Also, let $r=\lceil (n-l)/(l+1)\rceil$ and let  $L_{l}^{(1)},\dots,L_{l}^{(r)}$ be $r$ disjoint copies of $L_{l}$. Consider the bipartite comparability graph $B'$ between $L_{l}'=\bigcup_{i=1}^{r}L_{l}^{(i)}$ and $L_{l+1}$. Every vertex in $L_{l}'$ has degree $n-l$, while degree of every vertex in $L_{l+1}$ is $(l+1)r$. As $n-l\leq r(l+1)$, there is a complete matching $M_{2}$ from $L_{l+1}$ to $L_{l}'$. But then $M_{2}$ induces a subgraph $T$ of $B$ in which the degree of every vertex of $L_{l+1}$ is $1$, and the degree of every vertex in $L_{l}$ is at most $r$. Taking the union $M_{1}\cup T$, we get a subgraph of $B$ with minimum degree $1$ and maximum degree $r+1$. As $r+1<\lceil 1/c\rceil$, this finishes our proof.
\end{proof}

\begin{proof}[Proof of Lemma \ref{minmaxlemma}]
	Let us remind the reader that $\min B(k)=(\emptyset,\dots,\emptyset)$ and $\max B(k)=([n_{1}],\dots,[n_{1}])$.
	
	Let $S$ be the family of problematic elements of $B(k)$. Let $x,x'\in 2^{[n_{3}]}$ such that $x'\subset x$. First, we show that there is a collection $\mathcal{P}_{x,x'}$ of at least $2^{k-|P|}$ copies of $P$ in $\{x,x'\}\times S$ that cover $(x,\min B(k))$, and any two members of $\mathcal{P}_{x,x'}$ intersect only in $(x,\min B(k))$. Similarly, if $x\subset x'$, then there is a collection $\mathcal{P}_{x,x'}$ of at least $2^{k-|P|}$ copies of $P$ in $\{x,x'\}\times S$ that cover $(x,\max B(k))$, and any two members of $\mathcal{P}_{x,x'}$ intersect only in $(x,\max B(k))$. 
	
	We prove only the case $x'\subset x$, the other case is similar. Let $Q$ be a copy of $P$ in $2^{[0,|P|-1]}$ in which $\{0\}$ is minimal. For every $\alpha\subset [|P|,k]$ and $q\in Q$, define $q_{\alpha}$ as follows. If $q=\{0\}$, let $q_{\alpha}=(x,\min B(k))$. Otherwise, let $q_{\alpha}=(x_{0},y)$, where
	
	$$x_{0}=\begin{cases} x &\mbox{ if } 0\in q,\\
	                      x'&\mbox{ if } 0\not\in q,
	        \end{cases}$$	        
	 and for $i=1,\dots,k$,
	 $$y(i)=\begin{cases} [n_{1}] &\mbox{ if } i\in q\cup\alpha,\\
						   \emptyset &\mbox{ if } i\not\in q\cup\alpha.	 
	 \end{cases}$$
	  As $y(i)\in\{\emptyset,[n_{1}]\}$, $y$ is problematic. Also, it can be easily checked that $Q_{\alpha}=\{q_{\alpha}:q\in Q\}$  is a copy of $P$ which contains $(x,\min B)$.  Finally, if $\alpha\neq\alpha'$, then $Q_{\alpha}$ and  $Q_{\alpha'}$ only intersect in $(x,\min B(k))$.

	  Now let $c=0.1$ and apply Claim \ref{smallsets} to find a $P$-packing $\mathcal{P}_{1}'$ in $T(n_{3})$ that covers every set of size at most $cn_{3}$ or at least $(1-c)n_{3}$. Then $$\mathcal{P}_{1}=\{Q\times\{\min B(k)\}:Q\in\mathcal{P}_{1}\}\cup\{Q\times\{\max B(k)\}:Q\in\mathcal{P}_{1}\}$$
	  is a $P$-packing in $T(n_{3})\times \{\min B(k),\max B(k)\}$ that covers every element $(x,y)$ for which $y\in \{\min B(k),\max B(k)\}$, and $|x|\leq cn_{3}$ or $|x|\geq (1-c)n_{3}$. 
	  
	  It remains to find a $P$-packing $\mathcal{P}_{2}$ that covers the remaining elements of $T(n_{3})\times \{\min B(k),\max B(k)\}$ and is disjoint from $\mathcal{P}_{1}$. By Claim \ref{smalldegreegraph}, there exists a graph $G$ on $2^{[n_{3}]}$ such that the maximum degree of $G$ is at most  $2\lceil 1/c \rceil$, and for every $x\in 2^{[n_{3}]}$ satisfying $cn_{3}\leq |x|\leq (1-c)n_{3}$, there exist $x_{1},x_{2}\in 2^{[m]}$ such that $x_{1}\subset x\subset x_{2}$, and $xx_{1}$, $xx_{2}$ are edges of $G$. Now we cover the remaining elements $(x,y)\in T(n_{3})\times \{\min B(k),\max B(k)\}$, where $cn_{3}< |x|<(1-c)n_{3}$, one-by-one with the help of the following rule: if $x_{1}\subset x\subset x_{2}$ are such that $xx_{1}$ and $xx_{2}$ are edges of $G$, then we cover $(x,\min B(k))$ with a member of $\mathcal{P}_{x,x_{2}}$, and we cover $(x,\max B(k))$ with a  a member of $\mathcal{P}_{x,x_{1}}$.
	  
	   Suppose that at one step, we wish to add a copy of $P$ to $\mathcal{P}_{2}$ that covers $(x,\min B(k))$. As the maximum degree of $G$ is at most $2\lceil 1/c \rceil$, there are at most $2\lceil 1/c \rceil+1$ members of $\mathcal{P}_{2}$ that intersect $\{x\}\times S$, and at most $2\lceil 1/c \rceil+2$ members of $\mathcal{P}_{2}$ that intersect $\{x_{2}\}\times S$. (Here, we add $+1$ and $+2$ because there might be members of $\mathcal{P}_{2}$ already covering $(x,\max B(k))$, $(x_{2},\min B(k))$ and $(x_{2},\max B(k))$.) Hence, there are at most $4\lceil 1/c \rceil+3$ members of $\mathcal{P}_{2}$ that intersect any member of $\mathcal{P}_{x,x_{2}}$. Also, the members of $\mathcal{P}_{1}$ can intersect $\{x,x_{2}\}\times S$ in at most $4$ points, namely $\{x,x_{2}\}\times\{\min B(k),\max B(k)\}$. Therefore, if the inequality
	   $$|P|(4\lceil 1/c \rceil+3)+4< |\mathcal{P}_{x,x_{2}}|$$
	   holds, then there is a member of $\mathcal{P}_{x,x_{2}}$ that is disjoint from every member of $\mathcal{P}_{1}\cup\mathcal{P}_{2}$, and we add this copy of $P$ to $\mathcal{P}_{2}$. But this inequality clearly holds, as the left hand side is equal to $43|P|+4$, and the right hand side is at least $2^{k-|P|}$. We proceed similarly in the case we want to add a copy of $P$ to $\mathcal{P}_{2}$ that covers $(x,\max B(k))$.
	   
	   But then we are done as the $P$-packing $\mathcal{P}_{1}\cup\mathcal{P}_{2}$ satisfies the conditions.
\end{proof}

\subsection{Putting everything together}\label{sect:final}

In this subsection, we finish the proof of Theorem \ref{mainthm2}.

Let $n_{3}$ be any positive integer such that $2^{n_{3}}\geq N$, where $N$ is the number defined in the end of Subsection \ref{sect:modp}, and let $n=n_{2}+n_{3}$. We show that $T(n)$ has a $P$-partition that covers all but at most $|P|-1$ elements. We shall view $2^{[n]}$ as the cartesian product $C=2^{[n_{3}]}\times B(k)$, that is, the elements of $C$ are the pairs $(x,y)$, where $x\in 2^{[n_{3}]}$ and $y\in B(k)$. Also, let $C^{-}=C-\{\min C,\max C\}$.

First, let $\mathcal{P}_{1}$ be a $P$-packing described in Lemma \ref{minmaxlemma} that covers every element of $T(n_{3})\times \{\min B(k),\max B(k)\}$, and only covers elements $(x,y)\in C$ for which $y$ is problematic. 

Now for each $x\in 2^{[n_{3}]}$, consider $\{x\}\times B(k)$. Let $\mathcal{P}_{2}(x)$ be a $P$-packing in $\{x\}\times PR$ which covers every element $(x,y)$, where $y$ is problematic and $(x,y)$ is not covered by $\mathcal{P}_{1}$. There exists such a packing by Claim \ref{problematic}. Let $\mathcal{P}_{2}=\bigcup_{x\in 2^{[n_{3}]}} \mathcal{P}_{2}(x)$.

 We remind the reader an element of $B(k)$ is ordinary if it is not problematic, and $\mathcal{B}$ is the collection of absorbers in $B(k)$ not containing problematic elements, see Section \ref{sect:denseabsorber}. Let $A^{*}=\bigcup_{A\in \mathcal{B}} A$ be the family of elements in $B(k)$ contained in the absorbers, and let $O\subset B(k)$ be the family of ordinary points of $B(k)$ not contained in $A^{*}$. By Claim \ref{absorber_matching}, there exists an injection $\tau :O\rightarrow \mathcal{B}$ such that $\tau(y)$ completes $y\in O$. For $x\in 2^{[n_{3}]}$, let $O_{x}$ be the family of elements $y\in O$ such that $(x,y)$ is not covered by $\mathcal{P}_{2}$. Define the $P$-packing $\mathcal{P}_{3}(x)$ such that for each $y\in O_{x}$ we add a copy of $P$ which covers $(x,y)$ and is contained in $\{x\}\times (\tau(y)\cup\{y\})$. Let $\mathcal{P}_{3}=\bigcup_{x\in 2^{[n_{3}]}}\mathcal{P}_{3}(x)$.

So far, we have a $P$-packing $\mathcal{P}_{1}\cup\mathcal{P}_{2}\cup\mathcal{P}_{3}$ that covers every element of
$$C^{-}-(2^{[n_{3}]}\times A^{*}),$$
and for each $x\in 2^{[n_{3}]}$ and $A\in\mathcal{B}$, at most $|P|$ elements of $\{x\}\times A$ are covered. Let $m_{A}'$ be the number of elements of $2^{[n_{3}]}\times A$ covered and let $m_{A}=2^{n_{3}}|A|-m_{A}'$. Fix an arbitrary member $A_{0}$ of $\mathcal{B}$ and let $f:B(k)\rightarrow \mathbb{Z}_{|P|}$ be any function satisfying the following properties:

(1) $f(y)=0$, if $y\not\in A^{*}$,

(2) if $A\in\mathcal{B}-A_{0}$, then $\sum_{y\in A}f(y)=m_{A}$,

(3) $\sum_{y\in2^{[n_{3}]}}f(y)=0$.\\
One can easily construct such a function $f$: for example, take one element $y_{A}$ from every absorber $A\in\mathcal{B}$ and set $f(y_{A})=m_{A}$ if $A\neq A_{0}$, let $f(y)=0$ if $y\not\in\{y_{A}:A\in\mathcal{B}\}$, and let $f(y_{A_{0}})=-\sum_{y\in B(k)-y_{A_{0}}}f(y)$. By Claim \ref{modp}, $f$ is strongly realizable, which means that there exists a multiset $\mathcal{Q}'$ of copies of $P$ in $B(k)\setminus PR$ such that every $y\in B(k)$ is covered by $f(y)$ members of $\mathcal{Q}'$ modulo $|P|$. For each member of $Q\in\mathcal{Q}'$, pick a different element $x_{Q}\in 2^{[n_{3}]}$. We chose $n_{3}$ such that $2^{n_{3}}\geq N$, so this is possible.  Then $\mathcal{Q}=\{\{x_{Q}\}\times Q:Q\in\mathcal{Q}'\}$ is a $P$-packing in $C$. The members of $\mathcal{Q}$ are disjoint from the members of $\mathcal{P}_{1}$ and $\mathcal{P}_{2}$ as the members of $\mathcal{Q}'$ are disjoint from $PR$. However, members of $\mathcal{Q}$ might intersect members of $\mathcal{P}_{3}$. Nevertheless, by the definition of $f$, for each $A\in\mathcal{B}-A_{0}$ the number of elements of $2^{[n_{3}]}\times A$ covered by $\mathcal{P}_{3}\cup\mathcal{Q}$ (with multiplicity) is congruent to $2^{n_{3}}|A|$ modulo $|P|$.       

In what comes, we modify $\mathcal{P}_{3}$ to ensure that the members of $\mathcal{P}_{3}$ and $\mathcal{Q}$ are disjoint, while keeping the previously described congruency property of $\mathcal{P}_{3}\cup \mathcal{Q}$. Then, this congruency property is equivalent to the statement that the number of elements of $2^{[n_{3}]}\times A$ that are not covered by any member of $\mathcal{P}_{3}\cup \mathcal{Q}$ is divisible by $|P|$ for $A\in\mathcal{B}-A_{0}$.

 We shall modify  $\mathcal{P}_{3}$ in the following way. Suppose that $\{x\}\times P'\in \mathcal{P}_{3}(x)$ intersects $\{x\}\times Q\in\mathcal{Q}$. Then $P'$ can be written as $\{y\}\cup P^{-}$, where $y\in O_{x}\setminus PR$ and $P^{-}\subset \tau(y)$. Then there are two possibilities:

1. Suppose that $y\not\in Q$. As $\tau(y)$ completes $y$, there are at least $|P|$ copies of $P$ in $\tau(y)\cup\{y\}$ that contain $y$ and only intersect in $y$. Hence, we can choose at least one of these copies $P''$ that is disjoint from $Q$. Replace $\{x\}\times P'$ with $\{x\}\times P''$ in $\mathcal{P}_{3}$. Clearly, after this replacement, $\mathcal{P}_{3}$ contains the same number of elements of $2^{[n_{3}]}\times A$ for every $A\in \mathcal{A}$.

2. Now suppose that $y\in Q$. In this case, remove $\{x\}\times P'$ from $\mathcal{P}_{3}$. Let $x_{1},...,x_{s}\in 2^{[n_{3}]}$ be all the elements for which there exists $Q_{i}$ such that $\{x_{i}\}\times Q_{i}\in\mathcal{Q}$. By the definition of $f$, we have $s\equiv f(y)\equiv 0\Mod{|P|}$. Also, $\{x_{1}\}\times P',\dots,\{x_{s}\}\times P'$ are all members of $\mathcal{P}_{3}$ that we remove from $\mathcal{P}_{3}$. Hence, after the removal of $\{x_{1}\}\times P',\dots,\{x_{s}\}\times P'$ from $\mathcal{P}_{3}$, $\mathcal{P}_{3}$ contains $s(|P|-1)$ fewer elements of $2^{[n_{3}]}\times \tau(y)$, and the same number of elements of $2^{[n_{3}]}\times A$ for $A\in \mathcal{B}-\tau(y)$. Therefore, for every $A\in \mathcal{B}$ the number of elements of $2^{[n_{3}]}\times A$ covered by $\mathcal{P}_{3}$ did not change modulo $|P|$ after the modifications. 

\begin{figure}
	\begin{center}
		\includegraphics[scale=0.6,trim={0 6cm 0 0}]{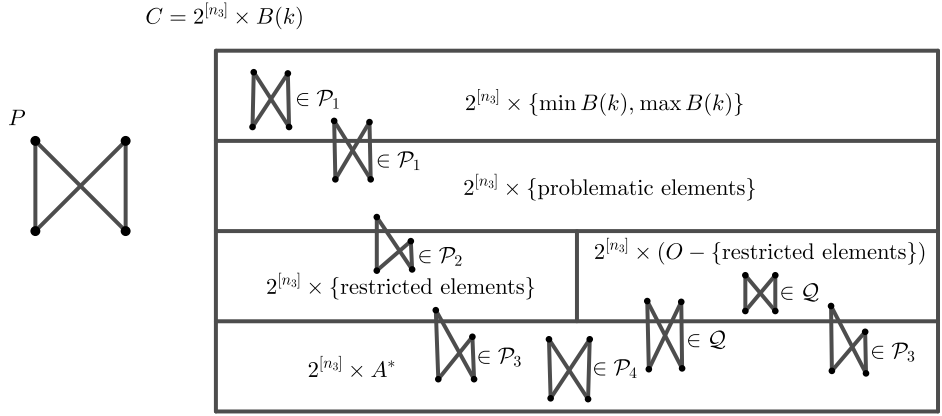}
	\end{center}
	\caption{An illustration of our final $P$-packing, where the drawings of $P$ represent the possible intersections of the copies of $P$ with our five different families. For example, every copy of $P$ intersecting $2^{[n_{3}]}\times \{\mbox{problematic elements}\}$ is either contained in the union of  $2^{[n_{3}]}\times \{\mbox{problematic elements}\}$ and $2^{[n_{3}]}\times \{\min B(k),\max B(k)\}$, or has exactly one element in $2^{[n_{3}]}\times \{\mbox{problematic elements}\}$ and the rest of its elements are in $2^{[n_{3}]}\times \{\mbox{restricted elements}\}$.}
	\label{image2}
\end{figure}

Let us summarize what we have so far. We constructed the $P$-packing $$\mathcal{P}^{*}=\mathcal{P}_{1}\cup\mathcal{P}_{2}\cup\mathcal{P}_{3}\cup Q$$ which has the following properties:

\begin{itemize}
	\item every element of $C^{-}-(2^{[n_{3}]}\times A^{*})$ is covered by $\mathcal{P}^{*}$,
	\item for every $x\in 2^{[n_{3}]}$ and $A\in \mathcal{B}$, $\mathcal{P}^{*}$ covers at most $2|P|$ elements of $\{x\}\times A$,
	\item for every $A\in \mathcal{B}-A_{0}$, the number of elements of $2^{[n_{3}]}\times A$ not covered by $\mathcal{P}^{*}$ is divisible by $P$.
\end{itemize}

But then, for every $A\in \mathcal{B}-A_{0}$, we can apply Claim \ref{absorber_product} to find a $P$-packing $\mathcal{P}_{A}$ that forms a $P$-partition of the uncovered elements $2^{[n_{3}]}\times A$. Also, there is $P$-packing $\mathcal{P}_{A_{0}}$ in the uncovered elements of $2^{[n_{3}]}\times A_{0}$ that covers all but at most $|P|-1$ elements. Set $\mathcal{P}_{4}=\bigcup_{A\in \mathcal{B}}\mathcal{P}_{A}$. 

We conclude the proof of Theorem \ref{mainthm2} by noting that $\mathcal{P}=\mathcal{P}^{*}\cup\mathcal{P}_{4}$ is a $P$-packing in $C$ that covers all but at most $|P|-1$ elements of $C^{-}\cong T(n)$.  See Figure \ref{image2} for an illustration of what kind of copies of $P$ the packing $\mathcal{P}$ is composed of.
	
\section{Concluding remarks}\label{sect:remarks}

In this paper, we proved that if  the poset $P$ has a unique minimum and maximum and size $2^{k}$, then $2^{[n]}$ has a $P$-partition for $n=\Omega(|P|^{8})$. However, we have no reason to believe that this bound on $n$ is sharp. In \cite{T}, it is shown that if $P$ is a chain, then the optimal bound is $\Theta(|P|^{2})$. Therefore, we propose the following conjecture.

\begin{conjecture}
    There exists a constant $c$ with the following property. Let $P$ be a poset with a unique minimum and maximum and size $2^{k}$. If $n\geq c|P|^{2}$, then $2^{[n]}$ has a $P$-partition.
\end{conjecture}

Also, in Theorem \ref{mainthm2}, one can backtrack our proof to find that $n_{0}(P)=2^{2^{O(|P|\log |P|)}}$. However, we conjecture that the right order of magnitude of $n_{0}(P)$ is also $|P|^{2}$.

\begin{conjecture}
	 There exists a constant $c$ with the following property. Let $P$ be a poset. If $n\geq c|P|^{2}$, then $2^{[n]}-\{\emptyset,[n]\}$ has a $P$-packing that covers all but at most $|P|-1$ elements.
\end{conjecture}

In Section \ref{sect:thm0}, we proved that if $P$ has a unique minimum and maximum and dimension $d$, then the grid $[2h]^{m}$ has a $P$-partition, where $m=O(d^{2})$, $|P|$ divides $h$, and $h$ is sufficiently large. One might wonder that what is the smallest dimension $m$ such that $[h]^{m}$ has a $P$-partition for some $h$. We claim that our proofs can be modified to show that $m\leq 2d-1$. Moreover, there are posets for which this bound is sharp: take $P=\frac{[2]^{d}}{[2]^{d}}$ for example. We omit the proofs.

 %Also, this bound is sharp: consider the poset $P=\frac{[2]^{d}}{[2]^{d}}$. Then the dimension of $P$ is $d$, and for any $p\in P$, if $\underline{P}_{p}=\{q\in P:q\leq_{P} p\}$ is the downset generated by $x$, and $\overline{P}_{p}=\{q\in P:p\leq_{P} q\}$ is the upset generated by $P$, then at least one of $\underline{P}_{p}$ or $\overline{P}_{p}$ has dimension $d$. Indeed, one of them contains a poset isomorphic to $[2]^{d}$. Suppose that $m=2d-2$ and let $x=(1,\dots,1,h,\dots,h)\in [h]^{m}$, where $x$ has $d-1$ coordinates equal to $1$, and $d-1$ coordinates equal to $h$. Then no copy of $P$ in $[h]^{m}$ can cover $x$. Otherwise, if $P'\subset [h]^{m}$ is a copy of $P$ containing $x$, where $x$ corresponds to $p\in P$, then consider one of $\underline{P}_{p}$ or $\overline{P}_{p}$, which has dimension $d$. Without loss of generality, suppose that $\underline{P}_{p}$ has dimension $d$. If $\underline{P}\subset P'$ corresponds to $\underline{P}_{p}$, then the first $d-1$ coordinates of every element of $\underline{P}$ is $1$, which means that $\underline{P}$ lives in the $d-1$-dimensional grid $\{1\}^{d-1}\times [h]^{d-1}$, contradiction.
 
 %This is in contrast with a result of \cite{GLT1}, where it is proved that for every $m$ there exists tile in $\mathbb{Z}$ which does not tile $\mathbb{Z}^{m}$.


\begin{thebibliography}{9}
	
\bibitem{BMS}
M. Bonamy, N. Morrison, A. Scott,
\emph{Partitioning the vertices of a torus into isomorphic subgraphs,}
arXiv:1710.07255
	
\bibitem{EGP}
P. Erd\H{o}s, A. Gy\'{a}rf\'{a}s, L. Pyber,
\emph{Vertex coverings by monochromatic cycles and trees,}
Journal of Combinatorial Theory B, 51 (1) (1991): 90--95.
	
	
\bibitem{Gray}
F. Gray,
\emph{Pulse code communication,}
United States Patent Number 2 632 058, March 17, 1953.
	
\bibitem{Gri}
J. R. Griggs,
\emph{Problems on Chain Partitions,}
Discrete Math. 72 (1988), 157--162.

\bibitem{Gru}
V. Gruslys, 
\emph{Decomposing the vertex set of a hypercube into isomorphic subgraphs,}
arXiv:1611.02021

\bibitem{GLT1}
V. Gruslys, I. Leader, T. S. Tan,
\emph{Tiling with arbitrary tiles,}
Proc. Lond. Math. Soc. 112 (6) (2016): 1019--1039.

\bibitem{GLT2} 
V. Gruslys, I. Leader, I. Tomon, 
\emph{Partitioning the Boolean lattice into copies of a poset},
arXiv:1609.02520	

\bibitem{GL}
V. Gruslys, S. Letzter,
\emph{Almost partitioning the hypercube into copies of a graph,}
arXiv:1612.04603
	
\bibitem{H}
P. Hall,
\emph{On Representatives of Subsets,}
J. London Math. Soc. 10 (1) (1935): 26--30.	

\bibitem{Hi}
T. Hiraguchi, 
\emph{On the dimension of orders,}
Sci. Rep. Kanazawa Univ. 4 (1955): 1--20.

\bibitem{L}
Z. Lonc,
\emph{Proof of a Conjecture on Partitions of a Boolean Lattice,}
Order 8 (1991): 11--21.

\bibitem{L2}
Z. Lonc,
\emph{Partitions of large Boolean lattices,}
Discrete Mathematics 131 (1994): 173--181.

\bibitem{MP}
A. Methuku, D. P\'{a}lv\"{o}lgyi,
\emph{Forbidden Hypermatrices Imply General Bounds on Induced Forbidden Subposet Problems,}
Combinatorics, Probability and Computing,  doi:10.1017/S0963548317000013.

\bibitem{S}
B. Sands,
\emph{Problem session, Colloquium on ordered sets,}
Szeged, Hungary (1985).

\bibitem{T}
I. Tomon,
\emph{Improved bounds on the partitioning of the Boolean lattice into chains of equal size,}
Discrete Mathematics, 339 (1) (2016): 333--343.

\bibitem{T2}
I. Tomon,
\emph{Almost tiling of the Boolean lattice with copies of a poset,}
Electronic Journal of Combinatorics 25 (1) (2018): P1.38.

\end{thebibliography}
\end{document}